\documentclass[12pt]{amsart}

\usepackage{url,  mathrsfs}
\usepackage{hyperref}
\usepackage{amsmath}
\usepackage{graphicx, 
epstopdf}
\usepackage{subfig}
\usepackage{color}
\usepackage{bbm}
\usepackage{subfloat}
\usepackage[draft]{fixme}
\usepackage{bm} 
\usepackage{bbm} 
\usepackage{epigraph}
\usepackage{endnotes}
\usepackage{ifpdf}
\usepackage{array}

\usepackage{multicol}
\usepackage{multirow}
\usepackage{graphicx}
\usepackage{tikz}

\newcommand{\Z}{\mathbb{Z}}
\newcommand{\R}{\mathbb{R}}
\newcommand{\F}{\mathbb{F}}
\newcommand{\C}{\mathbb{C}}
\renewcommand{\P}{\mathbb{P}}
\newcommand{\inv}{^{-1}}
\newcommand{\M}{\mathcal{M}}
\renewcommand{\L}{\mathcal{L}}
\newcommand{\V}{\mathcal{V}}

\newcommand{\1}{\mathbf{1}}
\newcommand{\B}{\mathcal{B}}
\renewcommand{\F}{\mathcal{F}}
\renewcommand{\H}{\mathcal{H}}
\newcommand{\Sym}{\textnormal{Sym}}

\newcommand{\Hom}{\textnormal{Hom}}
\newcommand{\tr}{\textnormal{tr}}
\newcommand{\ii}{\textnormal{i}}
\newcommand{\rb}{\iota_{\mathcal{B}}^{\vee}}

\theoremstyle{plain}
\newtheorem{thm}{Theorem} 
\newtheorem{prop}[thm]{Proposition}
\newtheorem{cor}[thm]{Corollary}
\newtheorem{lem}[thm]{Lemma}
\newtheorem{question}[thm]{Question}

\newtheorem*{thm*}{Theorem}

\theoremstyle{definition}
\newtheorem{Def}[thm]{Definition}
\newtheorem{example}[thm]{Example}
\newtheorem{remark}[thm]{Remark}

\numberwithin{equation}{section}
\numberwithin{thm}{section}

\DeclareMathOperator{\Gr}{Gr}

\DeclareMathOperator{\diag}{diag}
\DeclareMathOperator{\re}{Re}
\DeclareMathOperator{\im}{Im}
\DeclareMathOperator{\vecz}{vec}
 \def\dashmapsto{\mapstochar\dashrightarrow}

\usepackage[inner=1in,outer=1in,top=1in,bottom=1in]{geometry}

\makeatletter
\newcommand{\addresseshere}{%
  \enddoc@text\let\enddoc@text\relax
}
\makeatother

\begin{document}
\title{The Chow form of a reciprocal linear space}

 \author{Mario Kummer}
\address{Max Planck Institute for Mathematics in the Sciences, Leipzig, Germany} 
\email{kummer@mis.mpg.de}

 \author{Cynthia Vinzant}
\address{North Carolina State University, Raleigh, NC, USA}
\email{clvinzan@ncsu.edu}

\begin{abstract}
A reciprocal linear space is the image of a linear space under coordinate-wise inversion. 
These fundamental varieties describe the analytic centers of hyperplane arrangements 
and appear as part of the defining equations of the central path of a linear program. 
Their structure is controlled by an underlying matroid. This provides a large family of hyperbolic varieties, 
recently introduced by Shamovich and Vinnikov. 
Here we give a definite determinantal representation to the Chow form of a reciprocal linear space. 
One consequence is the existence of symmetric rank-one Ulrich sheaves on reciprocal linear spaces. 
Another is a representation of the entropic discriminant as a sum of squares. 
For generic linear spaces, the determinantal formulas obtained are closely related to the Laplacian of the complete graph and 
generalizations to simplicial matroids. 
This raises interesting questions about the combinatorics of hyperbolic varieties and connections with the positive Grassmannian. \\
\end{abstract}

\maketitle

The {\bf reciprocal linear space} of a $d$-dimensional linear space $\L \subseteq \C^n$ is the variety 
\[\L^{-1} \ \ = \ \ \overline{\{[x_1^{-1} : \hdots : x_n^{-1}] \ : \  x\in \L\cap (\C^*)^n\}} \ \ \subseteq \ \ \P^{n-1}(\C).\]
Proudfoot and Speyer \cite{PS} studied degenerations of the coordinate ring $\C[\L^{-1}]$ 
and proved that the degree of $\L^{-1}$ is the M\"obius invariant of the matroid $M(\L)$ associated with $\L$. 
Varchenko \cite{Var95} showed that if $\L$ is defined over $\R$ then all of the intersection points of $\L^{-1}$ 
with an affine linear space $\L^\perp + u$ for $u\in \R^n$ are real, where $\L^{\perp}$ denotes the orthogonal complement of $\L$. 
This fully real structure was exploited to study the central curve of a linear program \cite{DLSV12} and entropy maximization for log-linear models \cite{SSV13}.
In fact, this shows that the reciprocal linear space $\L^{-1}$ is \emph{hyperbolic} with respect to the linear space $\L^{\perp}$. 

The notion of a hyperbolic variety was recently introduced by Shamovich and Vinnikov \cite{SV} 
as a generalization of hyperbolic polynomials and hypersurfaces. 
The study of hyperbolic polynomials originated in the theory of partial differential equations \cite{Gar51} and has extended to optimization \cite{Guel97, Ren06},
combinatorics \cite{COSW04} and statistics \cite{MSUZ14}. 
In 2007, Helton and Vinnikov \cite{HV07} proved that every hyperbolic polynomial $f\in \R[x_0, x_1, x_2]_d$ has a definite symmetric
determinantal representation, $f = \det(\sum_i x_i A_i)$, where each $A_i$ is a real symmetric $d\times d$ matrix and $\sum_i v_i A_i$ 
is positive definite for some $v\in \R^3$.  In more variables, not all hyperbolic polynomials have such representations \cite{Bra11}. 
The challenge of testing whether a given hyperbolic polynomial has a definite determinantal representation and finding one if it exists is an active topic of research 
\cite{hanselka2014definite, kummer2013determinantal, us, TimundThom12, PV13}. See \cite{vppf} for a survey.

In \cite{SV}, the authors generalize the Helton--Vinnikov theorem to show that any hyperbolic curve 
has a definite determinantal representation, in the sense that its \emph{Chow form} has a definite determinantal representation with desired properties. 
Motivated by concepts from the theory of multivariate operators and multidimensional systems (vessels), they develop a theory of 
hyperbolic varieties of codimension $>1$ and their definite determinantal representations, which certify hyperbolicity.
These are intimately related to certain Ulrich sheaves supported on the variety, \cite{KuSh}.
There are still fundamental open question about the structure of hyperbolic varieties.  Reciprocal linear spaces provide a large 
class of explicit hyperbolic varieties on which to explore this developing theory.

The paper is organized as follows. In Section~\ref{sec:stablevars} we give basic definitions,
introduce a class of hyperbolic varieties that generalize hypersurfaces defined by stable polynomials, 
and prove that it is preserved under coordinate-wise inversion. 
In Section~\ref{sec:detBackground}, we recall the notion of Liv{\v{s}}ic-type determinantal representations and 
in Section~\ref{sec:detreps}, we show that they exist for reciprocal linear spaces.  Explicit formulas in the generic case are closely related to 
graphic and simplicial matroids, discussed in Section~\ref{sec:uniform}.  Section~\ref{sec:Hadamard} 
discusses relations to the Hadamard product of linear spaces, as studied in \cite{Hadamard}. 
The results of \cite{KuSh} and this paper imply that there exist rank-one Ulrich sheaves on reciprocal linear
spaces. In Section~\ref{sec:Ulrich}, we use this to prove the conjecture of \cite{SSV13} that 
the entropic discriminant is a sum of squares.

\bigskip \textbf{Acknowledgments.} 
We would especially like to thank Bernd Sturmfels for his guidance and introduction to this problem. 
This project started while both authors were attending the semester on ``Algorithms and Complexity in Algebraic Geometry'' 
at the Simons Institute for Theory and Computing.  We are grateful for discussions with Joseph Kileel, Radmila Sazdanovic, Eli Shamovich, and Victor Vinnikov. 
Over the course of this project, Mario Kummer was supported by the Studienstiftung des deutschen Volkes and Cynthia Vinzant received support from  
the National Science Foundation (DMS-1204447 and DMS-1620014).

\section{Hyperbolic varieties and the Positive Grassmannian}\label{sec:stablevars}

Here we give the technical definitions of hyperbolic varieties, as developed by \cite{SV}, and define a closely related 
notion of stability involving the positive Grassmannian.

\begin{Def}
 Let $L \subseteq \P^{n-1}(\C)$ be a linear subspace defined over the reals of projective dimension $c-1$.
 A  real quasi-projective variety $X \subseteq \P^{n-1}(\C)$ of codimension $c$ is {\bf hyperbolic}
 with respect to $L$  if $X \cap L = \emptyset$ and if for all real linear subspaces $L' \supset L$ of dimension $c$, 
 the intersection $X \cap L'$ consists only of real points. 
\end{Def}

When $X$ is a hypersurface ($c=1$), we recover the better known definition of a hyperbolic polynomial: 
A homogeneous polynomial $f \in \R[x_1,\ldots,x_n]_d $
is hyperbolic with respect to a point $L = e \in \P^{n-1}(\R)$ if $f(e)\neq 0$ and every line $L'$ containing $e$ 
meets $X = \V(f)$ only in real points. 
A closely related notion is \emph{stability}.  A polynomial $f \in \R[x_1, \ldots, x_n]$ is {\bf stable}
if $f(z) \neq 0$ whenever $\im(z)\in (\R_+)^n$. 
One can check that a homogeneous polynomial $f\in \R[x_1, \hdots, x_n]_d$ is stable if and only if it is hyperbolic with respect to 
every point in the positive orthant $\R_{+}^n$. 
Stable polynomials appear in analysis, combinatorics, and optimization \cite{wagnerSurvey} and were central in the recent proof of the Kadison--Singer conjecture \cite{MSS14}. 

A natural generalization of stability to varieties of higher codimension is to consider hyperbolicity with respect to 
an orthant in the Grassmannian of linear spaces of the correct dimension.  
The Grassmannian of $c$-dimensional linear subspaces of $\mathbb{A}^n$ is denoted by $\Gr(c,n)$.
We also identify $\Gr(c,n)$ with the set of $(c-1)$-dimensional linear subspaces of $\P^{n-1}(\C)$.

Consider the Pl\"ucker embedding $L \mapsto p(L)=(p_I(L) : I\in \binom{[n]}{c})$ 
of the Grassmannian $\Gr(c,n)$ in $\P^{\binom{n}{c}-1}$.  An orthant in the Grassmannian is the subset of real linear 
spaces whose Pl\"ucker coordinates have a prescribed sign pattern.  Specifically, for $\sigma \in \{\pm 1\}^{\binom{n}{c}}$, 
let $\Gr(c,n)^{\sigma}$ denote the subset of $\Gr(c,n)$ of $c$-dimensional subspaces $L \subseteq \R^n$ such that the products $\sigma_I p_I(L)$ are either all strictly positive or all strictly negative.
In particular, if $\sigma$ is the all-ones vector, then $\Gr(c,n)^{\sigma}$ is the \textit{positive Grassmannian},
studied in the theory of totally positive matrices \cite{lusztig} and scattering amplitudes in string theory \cite{scattering}.
This gives a natural generalization of stable polynomials to varieties of higher codimension.

\begin{Def}
A real, quasi-projective variety $X \subseteq \P^{n-1}(\C)$ of codimension $c$ is 
{\bf $\sigma$-stable} for $\sigma\in\{\pm1\}^{\binom{n}{c}}$ if $X$ is hyperbolic with respect to every $L\in \Gr(c,n)^\sigma$.
\end{Def}

Operations that preserve stability of hypersurfaces are well-understood \cite{stabilityPreservers, wagnerSurvey}.
It would be interesting to give a similar characterization of operation preserving stability of varieties.  
One operation that does extend to varieties of higher codimension is 
coordinate-wise inversion. In order to prove this, we need another characterization of hyperbolicity. 

\begin{prop}\label{prop:hypEquiv}
 Let $X \subseteq \P^{n-1}(\C)$ be a real quasi-projective variety of codimension $c$ and 
$L\subset \P^{n-1}(\C)$ be a real projective linear space of dimension $c-1$. 
 Let $\hat{L}(\R) \subseteq \R^{n}$ be the real points in the affine cone $\hat{L}$ over $L$.
 Then the following are equivalent:
 \begin{enumerate}
  \item $X$ is hyperbolic with respect to $L$.
  \item For all  $ a \in \R^{n}$ and $0 \neq b \in \hat{L}(\R)$, and we have that $[a+\textnormal{i} b ] \not\in X$.
 \end{enumerate}
\end{prop}

\begin{proof}
 ($\Rightarrow$) Let $ a \in \R^{n}$ and $0 \neq b \in \hat{L}(\R)$ .  The point $p=[a+\textnormal{i} b ]$ belongs to 
 the projective linear space spanned by $L$ and $a$.  Suppose that $p$ belongs to $X$. 
Since $X$ is hyperbolic with respect to $L$, it follows that $p\in \P^{n-1}(\R)$, \emph{i.e.} for some $\lambda = \lambda_r + \textnormal{i}\lambda_i \in \C^*$ with $\lambda_r$ and $\lambda_i\in\R$, the vector $\lambda (a+\textnormal{i}b)$ is real.  Taking imaginary parts shows that 
$\lambda_i a + \lambda_r b =0$ and $a,b$ are linearly dependent.  The point $p$ then belongs to $ L$, which contradicts $L \cap X = \emptyset$.
 
 ($\Leftarrow$) Let $[c+\textnormal{i}d]\in L$ with $c,d\in\R^n$. Since $L$ is a real projective linear space, we have $d\in\hat{L}(\R)$. 
Taking $a=c$ and $b=d$ shows that $[c+\textnormal{i}d]\not\in X$, meaning that $L \cap X = \emptyset$.
Now suppose $L' \supseteq L$, where $L'$ is a real linear subspace of projective dimension $c$, and consider a point 
$p$ in the intersection $X \cap L'$.
We can write $L'$ as the span of $L$ and some real vector $a\in \R^n$, meaning that the point $p$ equals $[\lambda a + b]$ 
for some $b \in \hat{L}$ and $\lambda \in \C$. 
Since $L\cap X$ is empty, we have that $\lambda\neq 0$. We can rescale $\lambda$ and $b$ so that $\lambda =-\textnormal{i}$ and 
$p = [a + \lambda^{-1}b] = [a+\textnormal{i}b]$.    In particular, this writes $p = [(a - \im(b)) + \textnormal{i} \re(b)]$. 
Since $p\in X$, $(a - \im(b))\in \R^n$, and $\re(b) \in \hat{L}(\R)$, it must be that $\re(b)=0$ and $p\in \P^{n-1}(\R)$. 
\end{proof}

This new characterization can be used to show that the class of $\sigma$-stable varieties is closed under coordinate-wise inversion.  Formally, for any quasi-projective variety $X\subset \P^{n-1}(\C)$, let $X^{-1}$ denote its image under the rational map $[x_1: \hdots: x_n]\dashmapsto [x_1^{-1}: \hdots: x_n^{-1}]$.

\begin{prop}\label{prop:inversion}
Let $X \subseteq \P^{n-1}(\C)$ be an irreducible real, projective variety of codimension $c$ not contained in any 
coordinate hyperplane and let $\sigma \in \{\pm 1\}^{\binom{n}{c}}$. 
  Then $X$ is $\sigma$-stable if and only if its reciprocal $X^{-1}$ is $\sigma$-stable.
\end{prop}
\begin{proof}
By Proposition~\ref{prop:hypEquiv}, $X$ is $\sigma$-stable if and only if $X\cap \mathcal{S}_\sigma$ is empty, where 
\[ \mathcal{S}_{\sigma} \ = \ \{ [a+\textnormal{i}b]  \ : \ a\in \R^n, \  0\neq b\in \hat{L}(\R) \text{ for some } L \in \Gr(c,n)^{\sigma}\} \  \subset  \ \P^{n-1}(\C).\]
Let $U$ be the set of points in $\P^{n-1}(\C)$ with all coordinates non-zero. 
We claim that $\mathcal{S}_{\sigma}\cap U $ is invariant under coordinate-wise inversion. First, note that the multiplicative group 
$(\R_+)^n$ acts by coordinate-wise multiplication on $\C^{n}$, and the induced action on $\Gr(c,n)$ preserves $\Gr(c,n)^\sigma$.
It follows that $\mathcal{S}_{\sigma}$ is also invariant under this action of $(\R_+)^n$. It is also closed under complex conjugation.  
Coordinate-wise inversion can be realized as a composition of these actions. 
Specifically, if $p = re^{\ii \theta}\neq 0$, then $p^{-1} = r^{-1}e^{-\ii \theta} = r^{-2} \cdot \overline{p}$.  Therefore $\mathcal{S}_{\sigma}\cap U $ is closed under coordinate-wise inversion. 

If $X$ is $\sigma$-stable, then  $ X \cap \mathcal{S}_{\sigma}$ is empty. 
Since  $\mathcal{S}_{\sigma}\cap U $ is closed under coordinate-wise inversion, it follows that 
$ X^{-1} \cap \mathcal{S}_{\sigma} \cap U $ is empty.  Note that $\mathcal{S}_{\sigma} $ is open in the Euclidean topology on $\P^{n-1}(\C)$. 
Since $X$ is irreducible and not contained in a coordinate hyperplane, 
$X^{-1}\cap U$ is open and dense in the closed set $X^{-1}$ (with respect to both the Euclidean and the Zariski topology).
 It follows that $\mathcal{S}_{\sigma} $ intersects $X^{-1}\cap U$ if and only if $\mathcal{S}_{\sigma} $ intersects $X^{-1}$. 
Thus $ X^{-1} \cap \mathcal{S}_{\sigma} = \emptyset$ and $X^{-1}$ is $\sigma$-stable. 
Since $X$ is irreducible and not contained in a coordinate hyperplane, $X$ equals $(X^{-1})^{-1}$. The converse follows. 
\end{proof}

From this theory, we can see that $\L^{-1}$ is hyperbolic with respect to $\mathcal{L}^{\perp}$ and 
that $\L^{-1}$ is $\sigma$-stable for any $\sigma$ describing the sign pattern of the non-zero Pl\"ucker coordinates 
of $\L^{\perp}$.

\begin{cor}
 The reciprocal linear space $\mathcal{L}^{-1}$ is hyperbolic with respect to $\mathcal{L}^{\perp}$.
 Furthermore, if  $\L^{\perp}$ belongs to the Euclidean closure of $\Gr(n-d,n)^{\sigma}$, then $\L^{-1}$ is $\sigma$-stable. 
\end{cor}

\begin{proof}
Suppose that $\sigma\in \{\pm 1\}^{\binom{n}{n-d}}$ satisfies $\sigma_I = {\rm sign}(p_I(\L^{\perp}))$ 
whenever $p_I(\L^{\perp})\neq 0$. We will show that $\L$, and thus $\L^{-1}$, is $\sigma$-stable. 
Note that $M \mapsto M^{\perp}$ defines a map from $\Gr(n-d,n)^{\sigma}$ to $\Gr(d,n)^{\sigma'}$
where $\sigma'_I = (-1)^{s(I)} \sigma_{[n]\backslash I}$ and $s(I) =\frac{1}{2}d(d+1)+ \sum_{i \in I} i$. 
Since $\sigma$ agrees in sign with the non-zero 
Pl\"ucker coordinates of $\L^{\perp}$, then $\sigma'$ agrees in sign with the non-zero Pl\"ucker coordinates of $\L$. 

The linear variety $\L$ is hyperbolic with respect to a real projective linear space $M$ if and only if $\L\cap M$ is empty. 
Let $M\in \Gr(n-d, n)^{\sigma}$. By the Cauchy--Binet Theorem, 
$ \L\cap M\neq \emptyset$ if and only if $\sum_{I\in \binom{[n]}{d}} p_I(\L)p_I(M^{\perp})$ equals zero.
Since the non-zero Pl\"ucker coordinates of $\L$ and $M^{\perp}$ agree in sign, the sum $\sum_{I} p_I(\L)p_I(M^{\perp})$ 
is non-zero and $\L\cap M$ is empty. This shows that $\L$ is $\sigma$-stable. By Proposition~\ref{prop:inversion},
it follows that $\L^{-1}$ is $\sigma$-stable. 

Finally, $\L^{-1}$ is hyperbolic with respect to a linear space $M$ in the closure of $\Gr(n-d,n)^\sigma$  
as long as $\L^{-1}\cap M$ is empty by \cite[Thm. 3.10]{SV}.  By \cite[Lemma 27]{SSV13}, the intersection $\mathcal{L}^{-1} \cap \mathcal{L}^{\perp}$ is empty, 
meaning that $\L^{-1}$ is hyperbolic with respect to $\L^{\perp}$.  
\end{proof}

The theory of stable polynomial has rich connections with combinatorics and matroid theory, \cite{BrandenHPP, COSW04, Gurvits}. 
In Sections~\ref{sec:detreps} and \ref{sec:uniform}, we associate to $\L^{-1}$ a multiaffine stable polynomial in 
Pl\"ucker coordinates on $\Gr(n-d,n)$, 
whose support consists of the bases of a matroid on $\binom{n}{d}$ elements. 
It would be very interesting to generalize this to arbitrary stable varieties. 

\begin{question} Is there a polymatroid associated to a stable variety $X\subset \P^{n-1}(\C)$?\end{question}

\section{Background on determinantal representations and Chow forms}\label{sec:detBackground}
In this section we recall from \cite{SV} the definition and some properties
of Liv{\v{s}}ic-type determinantal representations.  First, we fix some notation for working in the Grassmannian. 

Let $1 \leq d \leq n$ and take $e_1, \ldots, e_n $ to be the standard basis for $\R^n$.
For each subset $I =\{i_1, \ldots, i_d\}$ of $[n]$ with $i_1 < \cdots < i_d$, 
we denote $e_I=e_{i_1} \wedge \cdots \wedge e_{i_d}$.
The collection of wedge products $\{e_I:\, I \in \binom{[n]}{d}\}$ forms a basis of $\bigwedge^d \R^n$.
For a linear space $L = {\rm span}\{v_1, \hdots, v_d\}$ in $\Gr_{\R}(d,n)$, 
we can express $v_1 \wedge \hdots \wedge v_d$ as $\sum_{I} p_I(L)e_I \in \bigwedge^d\R^n$.  
The coefficients $p(L) = (p_I(L))_{I\in \binom{n}{d}}$, known as the Pl\"ucker coordinates of $L$, 
are independent of the basis $\{v_1, \hdots, v_d\}$ of $L$ up to global scaling. 
The map that sends ${\rm span}\{v_1, \hdots, v_d\} $ to $[v_1 \wedge \hdots \wedge v_d]$
gives the Pl\"ucker embedding of the Grassmannian $\Gr_{\R}(d,n)$ into $ \P(\bigwedge^d \R^n)\cong \P^{\binom{n}{d}-1}(\R)$.

Given a $(d-1)$-dimensional variety $X\subset \P^{n-1}(\C)$, the collection of linear spaces $L$ that intersect 
$X$ form a hypersurface in $\Gr(n-d,n)$.  The element of the coordinate ring of $\Gr(n-d,n)$ 
defining this hypersurface is known as the  {\bf Chow form} of $X$, after \cite{Chow}. The Chow form can by represented by a polynomial 
in the Pl\"ucker coordinates of $L$ whose degree equals $\deg(X)$.  See, for example, \cite{ChowIntro}.  
We compute the Chow form of $\L^{-1}$ in Section~\ref{sec:detreps}.

It will often be convenient to identify $\bigwedge^{d} (\R^{n})$ with $(\bigwedge^{n-d} \R^{n})^{\vee}$ as follows. 
We can identify $\R$ with $\bigwedge^n \R^n$ via the isomorphism $\lambda \mapsto \lambda \cdot (e_1 \wedge \cdots \wedge e_n)=\lambda \cdot e_{[n]}$.
The resulting pairing
 \[\bigwedge\nolimits^d \R^n \times \bigwedge\nolimits^{n-d} \R^n  \  \to  \ \bigwedge\nolimits^n \R^n \ \ \text{ given by } \ \ 
 (\alpha, \beta) \mapsto \alpha \wedge \beta\]
provides an identification of $\bigwedge\nolimits^{n-d} \R^n$ with the dual space 
$(\bigwedge\nolimits^{d} \R^n)^{\vee}$.
This identifies the dual basis $\{e_I^*: I\in \binom{[n]}{d}\}$ of $(\bigwedge\nolimits^{d} \R^n)^{\vee}$ with the basis 
$\{\pm e_J : J \in \binom{[n]}{n-d}\}$ of $\bigwedge\nolimits^{n-d} \R^n $, specifically 
\[ e_I^* = (-1)^{s(I)}e_{[n]\backslash I}  \ \  \text{ where } \ \ s(I) = \frac{1}{2}d(d+1) + \sum_{i\in I}i.\]

Consider $\varphi \in \bigwedge^{d} \R^{n} \otimes \Sym_k(\R)$, where $\Sym_k(\R)$ denotes the space of $k\times k$ real symmetric matrices. 
Via the identification of $\bigwedge^{d} (\R^{n})$ with $(\bigwedge^{n-d} \R^{n})^{\vee}$, 
we can identify $\varphi$ with a $k \times k$ matrix, whose entries are linear forms on $\bigwedge^{n-d} \R^{n}$. 
The operator $\varphi$ is {\bf nondegenerate} if there exist vectors $v_1, \ldots, v_{n-d} \in \R^n$ defining 
$\beta = v_1 \wedge \ldots \wedge v_{n-d} \in \bigwedge^{n-d} \R^{n}$
for which the matrix $\varphi \wedge \beta \in \Sym_k(\R)$ is invertible.

\begin{Def} \label{dfn:Livsic-type_det_rep}
Let $X \subseteq \P^{n-1}(\C)$ be a real projective variety of dimension $d-1$. 
We say that $\varphi \in \bigwedge^{d} \R^{n} \otimes \Sym_k(\R)$
is a {\bf Liv{\v{s}}ic-type determinantal representation} of $X$ if 
$\varphi$ is nondegenerate and $X$ equals the set of points $p \in \P^{n-1}(\C)$ for which $\varphi \wedge p$ has a non-trivial 
kernel, when considered as a linear map from $\C^k$ to $\bigwedge^{d+1}\C^{n} \otimes \C^k$.
We say that $\varphi$ is {\bf definite} at a linear space $L = {\rm span}\{v_1, \hdots,  v_{n-d}\} \subset \R^n$ 
if the matrix $\varphi \wedge v_1\wedge \hdots \wedge v_{n-d}\in \Sym_k(\R)$ is (positive or negative) definite. 
\end{Def}

As observed in \cite{SV},  Liv{\v{s}}ic-type determinantal representations of $X\subset \P^{n-1}(\C)$
of the smallest size give determinantal representations of the Chow form of $X$. 

\begin{prop}\label{prop:livsicchow}
If $\varphi\in\bigwedge^{d} \R^{n} \otimes \Sym_k(\R)$ is a Liv{\v{s}}ic-type determinantal representation of 
a variety $X\subset \P^{n-1}(\C)$ of dimension $d-1$ and degree $k$, then the determinant of 
$\varphi\wedge \beta$, considered as a polynomial in $\R[\beta_{I} : I\in \binom{[n]}{n-d}]$, defines the Chow form of $X$. 
\end{prop}

\begin{proof}
By definition, for any point $p\in X$, there is a non-zero vector $ u_p \in \C^k$ in the kernel of the linear 
map given by $\varphi\wedge p$.  It follows that for any $\gamma \in \bigwedge^{n-d-1}\C^n$, the vector $u_p$ 
belongs to the kernel of the matrix $\varphi\wedge p \wedge \gamma \in \Sym_k(\C)$ and 
$\det(\varphi\wedge p \wedge \gamma) = 0$. 

If $\beta \in \bigwedge^{n-d}\R^n$ is the vector of the 
Pl\"ucker coordinates of any linear space $L$ containing the point $p$
we can write $\beta$ as $p\wedge  v_2 \wedge \hdots \wedge v_{n-d}$ for some $v_j \in \R^n$. 
In particular, $\det(\varphi\wedge \beta) =0$. 
Therefore the hypersurface in $\Gr(n-d,n)$ defined by $\det(\varphi\wedge \beta)=0$ contains the hypersurface 
defined by the vanishing of the Chow form of $X$. Since both hypersurfaces have degree $k$, they must be equal.
\end{proof}

\begin{remark}
 Let $X \subseteq \P^{n-1}(\C)$ be a real $(d-1)$-dimensional variety of degree $k$.
If $X$ has a Liv{\v{s}}ic-type determinantal representation
 $\varphi \in \bigwedge^{d} (\R^{n}) \otimes \Sym_k(\R)$ that is definite at a linear subspace
 $L$, then $X$ is hyperbolic with respect to $L$, cf. \cite[Prop. 3.12]{SV}.
\end{remark}

\begin{example}($d=2$, $n=4$, $k=3$) 
Consider $\varphi \in \bigwedge^{2} \R^{4} \otimes \Sym_3(\R)$ given by
\[
\varphi = 
\begin{pmatrix}
e_{12}& e_{13}& e_{14} \\ 
e_{13} & e_{14}+ e_{23} & e_{24} \\  
e_{14} & e_{24} & e_{34}
\end{pmatrix}, 
\ \ \text{ with } \ \ 
\varphi \wedge \beta = 
\begin{pmatrix}
\beta_{34} & -\beta_{24}& \beta_{23} \\ 
-\beta_{24} & \beta_{23}+ \beta_{14} & -\beta_{13} \\  
\beta_{23} & -\beta_{13} & \beta_{12}
\end{pmatrix}
\]
for all $\beta = \sum_{i<j} \beta_{ij} e_{ij} \in \bigwedge^2\R^4$.  Since the matrix 
$\varphi\wedge e_2\wedge e_3$ has full rank, $\varphi$ is nondegenerate. 
Furthermore, $\varphi$ has the property that for $v_1, v_2\in \C^4$, the matrix 
$\varphi\wedge v_1 \wedge v_2 \in \Sym_3(\C)$ drops rank if and only if the projective line 
spanned by $v_1$ and $v_2$ intersects the twisted cubic, $X = \{[s^3:s^2t:st^2:t^3]:[s:t]\in \P^1(\C)\}$.  
Not only does the matrix $\varphi\wedge v_1 \wedge v_2$ drop rank, 
but its kernel depends only on the intersection point $p$ of the line with $X$. In this case, 
$\varphi \wedge p$, considered as the linear map $\C^3 \rightarrow \bigwedge^3\C^4 \otimes \C^3$ given by 
\[
\begin{pmatrix} u_1 & u_2 & u_3 \end{pmatrix}
 \ \ \mapsto \ \
\begin{pmatrix}
e_{12}\wedge p& e_{13}\wedge p& e_{14}\wedge p \\ 
e_{13}\wedge p & e_{14}\wedge p+ e_{23}\wedge p & e_{24}\wedge p \\  
e_{14} \wedge p& e_{24}\wedge p & e_{34}\wedge p
\end{pmatrix} \begin{pmatrix} u_1 \\  u_2 \\  u_3 \end{pmatrix}, 
\]
has a nontrivial kernel, namely if $p=[s^3:s^2t:st^2:t^3]$, then $(s^2,st,t^2)$ is in the kernel.
Therefore $\varphi$ is a Liv{\v{s}}ic-type determinantal representation of the twisted cubic. 
The determinant $\det(\varphi\wedge\beta)$ is the Chow form of $X$, which is classically known to be 
the resultant of two binary cubics \cite[\S1.2]{ChowIntro}.  A reciprocal line in $\P^3$ is also a rational cubic curve, and we find
a related formula for the resultant in Corollary~\ref{cor:resultant}. \end{example}

\section{Determinantal representations of reciprocal linear spaces}\label{sec:detreps}
The goal of this section is to construct a definite Liv{\v{s}}ic-type determinantal representation 
of a reciprocal linear space, which gives a determinantal representation to its Chow form. 

Define the \textbf{support} ${\rm supp}(\alpha) \subseteq \binom{[n]}{d}$ of a non-zero point $\alpha =\sum_I \alpha_I e_I\in \bigwedge^d \R^n$ to be the 
set of $I \in \binom{[n]}{d}$ for which $\alpha_I \neq 0$. Let $\1=\sum_{i=1}^n e_i$.
For any subset $\B \subseteq \binom{[n]}{d}$ consider
the linear subspace  $\H_{\B}$ of $\bigwedge^d \R^n$ defined as 
\begin{equation}\label{eq:HB}
 \H_{\B} \ = \ {\rm span}\left\{\gamma \wedge \1 \ : \, \gamma \in \bigwedge\nolimits^{d-1} \R^n \right\} \ \cap  \ 
\bigl\{\alpha : {\rm supp}(\alpha) \subseteq  \B\bigl\}.
\end{equation}
The  vector of Pl\"ucker coordinates $p(L)$  of a linear space $L$ belongs to $\H_{\B}$ if and only if $L$ contains the vector $\1$ and ${\rm supp}(p(L)) \subseteq \B$.  This condition is convenient when considering reciprocal linear spaces, since $\1 \in \L$ if and only if $1\in \L^{-1}$. 

Consider the inclusion $\iota_{\B}: \H_{\B} \hookrightarrow \bigwedge^d \R^n$ and its dual 
$\rb: (\bigwedge^{d} \R^n)^{\vee} \to \H_{\B}^{\vee}$, which restricts a linear form on $\bigwedge^{d} \R^n$ to a linear form on the subspace 
$\H_{\B}$. For any $\alpha \in \bigwedge^d\R^n$ with $\B = {\rm supp}(\alpha)$, we will define a linear map 
$\varphi_{\alpha}: (\bigwedge^{d} \R^n)^{\vee} \to \Sym(\H_\B)$, where $\Sym(\H_\B) \subset \H_\B^{\vee}\otimes  \H_\B^{\vee}$  
denotes the space of symmetric bilinear forms on $\H_\B$. 
It suffices to define $\varphi_{\alpha}$ on the basis $\{e_{I}^* : I\in \binom{[n]}{d}\}$ of $(\bigwedge^{d} \R^n)^{\vee}$:
\begin{equation}\label{eq:phiAlpha}
\varphi_\alpha(e_{I }^*) \  = \ \begin{cases}
                          0 & \textrm{ if }I \notin  \B,  \\
                         \alpha_{I}^{-1} \cdot \rb(e_{I}^*) \otimes \rb(e_{I}^*), & \textrm{ if } I \in \B.
                         \end{cases}
\end{equation}
The linear map $\varphi_{\alpha}$ sends $\beta \in (\bigwedge^{d} \R^n)^{\vee} \cong \bigwedge^{n-d} \R^n $ 
to a symmetric bilinear form $\varphi_{\alpha}(\beta)$ on $\H_{\B}$.  Indeed, if $\beta \in\bigwedge^{n-d} \R^n $,  
then for any $p =  \sum_{I\in \B }p_Ie_I$, $ q =  \sum_{I\in \B }q_Ie_I \in \H_\B$, we have that
\begin{equation} \label{eq:phiBilinear}
\varphi_{\alpha}(\beta)(p,q) 
\ \ = \ \ \sum_{I\in \B} \alpha_I^{-1}\cdot (e_I \wedge \beta) \cdot p_I \cdot q_I  
\ \ \in  \ \R.
\end{equation}
This identifies $\varphi_{\alpha}$ with an element of $\bigwedge^{d} \R^n \otimes \Sym(\H_\B)$. 
If $k = \dim(\H_{\B})$, then picking a basis for $\H_{\B}$ writes $\varphi_{\alpha}$ as an element of 
$\bigwedge^{d} \R^n\otimes \Sym_k(\R) $. 
\

If $[\alpha] = p(\L)$ for some real linearspace $\L\in \Gr(d,n)$, 
then  $\varphi_\alpha$ is non-degenerate and positive definite at the the orthogonal complement $\L^{\perp}$.  Indeed, 
 $[\beta] = p(\L^{\perp})\in \P(\bigwedge^{n-d}\R^n)$ 
 is identified with the point $[\alpha^*] = [\sum_I \alpha_I e_I^*]$ in $\P((\bigwedge^d \R^n)^\vee)$, and 
  \[ \varphi_{\alpha}(\beta)  \ \ = \ \   \varphi_{\alpha}(\alpha^*) \ \ =  \ \  \sum_{I \in \B} \rb(e_I^*) \otimes  \rb(e_I^*),\]
 which is  positive definite on $\H_{\B}$. Therefore  $\varphi_\alpha$ is definite at $\L^{\perp}$, and, in particular, non-degenerate.
We will see  that $\varphi_\alpha$ is actually a determinantal representation of $\L^{-1}$.

\begin{thm}\label{thm:detrep}
 Let $\L \subseteq \P^{n-1}(\C)$ be a real linear subspace of dimension $d-1$ that is not contained
 in any coordinate hyperplane. Let $[\alpha] = p(\L) \in \P(\bigwedge^d \R^n)$ be the vector of Pl\"ucker coordinates of $\L$ and $\B={\rm supp}(\alpha)$. 
 Then $k= \dim(\H_\B) =  \deg(\L^{-1})$ and the map $\varphi_\alpha$ defined in \eqref{eq:phiAlpha}, as considered of an element of $\bigwedge^d \R^n \otimes \Sym_k (\R)$,
 is a symmetric Liv{\v{s}}ic-type determinantal representation of $\L^{-1}$ that is definite at the linear space $\L^{\perp} \in \Gr(n-d,n)$.
\end{thm}

We will build up to the proof of Theorem~\ref{thm:detrep} on page~\pageref{proof:detrep}. 
The general strategy is to find a kernel of $\varphi_{\alpha}(\beta)$ when $\alpha, \beta$ 
corresponding to a linear spaces $\L, \M$ both containing the all-ones vector $\1$.
This a convenient condition because $\1\in \L$ if and only if $\1\in \L^{-1}$. 
We will then use the action of the torus $(\R^*)^n$ to translate to a kernel of $\varphi_{\alpha}(\beta)$
when $\L^{-1}$ and $\M$ intersect in an arbitrary point. 
Some matroid theory is also needed to prove $\dim(\H_\B) =  \deg(\L^{-1})$.

\begin{lem}\label{lem:wedge1}
 Let $\B\subseteq \binom{[n]}{d}$ and suppose that $\alpha \in \H_{\B} \subset \bigwedge^d\R^n$ where $\B = {\rm supp}(\alpha)$.
 Then for every $\delta \in \bigwedge^{n-d-1} \R^n$, $\alpha$ belongs to the kernel of the bilinear form 
 $\varphi_{\alpha}(\delta \wedge \1)$.
 \end{lem}

\begin{proof}
For any $\beta \in \bigwedge^{n-d}\R^n$, $\varphi_\alpha(\beta)$ is a symmetric bilinear form on $\H_\B$. 
Plugging $\alpha$ into the first coordinate results in a linear form on $\H_{\B}$ given by  $q \mapsto q\wedge \beta$. 
Indeed, for any element $q = \sum_{I\in \B}q_Ie_I$ of $\H_\B$, we have
\[ \varphi_{\alpha}(\beta)(\alpha, q)  \ \ = \ \ \sum_{I\in \B} \alpha_I^{-1} \cdot (e_I\wedge \beta) \cdot \alpha_I \cdot q_I 
\ \ = \ \ \sum_{I\in \B} (e_I\wedge \beta) \cdot q_I  \ \   = \ \  q \wedge \beta.\]
Now suppose $\beta = \delta \wedge \1$ for some $\delta \in \bigwedge^{n-d-1}\R^n$.
Since $\H_\B$ is contained in the the span of vectors $\{\gamma\wedge \1 :\gamma \in \bigwedge^{d-1}\R^n\}$, 
we see that $q \wedge \beta = 0$ for all $q\in \H_\B$.  As a linear form on $\H_\B$, $q\mapsto \varphi_\alpha(\beta)(\alpha, q)$ is identically zero, 
meaning that $\alpha$ belongs to the kernel of $\varphi_{\alpha}(\beta)$. 
\end{proof}

Consider the map  ${\rm diag}_w: \R^n \rightarrow \R^n$ given by scaling the $i$th coordinate by $w_i$, $e_i \mapsto w_ie_i$. 
This extends to the linear map ${\rm diag}_w: \bigwedge \R^n \rightarrow  \bigwedge \R^n$ given by $e_I\mapsto (\prod_{i\in I}w_i )e_I$.

\begin{lem}\label{lem:wedge2}
Let $[\alpha] = p(\L) \in \P(\bigwedge^d \R^n)$ be the vector of Pl\"ucker coordinates of $\L$ and $\B={\rm supp}(\alpha)$. 
Suppose that $(w_1^{-1},\ldots, w_n^{-1}) \in \L$ for some $w \in (\R^*)^n$. 
 Then for every $\delta \in \bigwedge^{n-d-1} \R^n$, 
 $\diag_w(\alpha)\in \H_\B$ belongs to the kernel of the bilinear form $\varphi_{\alpha}(\delta \wedge w)$.
\end{lem}

\begin{proof}
Let $\alpha' = \diag_{w}(\alpha)$. First let us check that $\alpha'\in \H_\B$.
First note that $\alpha$ and $\alpha'$ have the same support, which is $\B$. 
Since $w^{-1}\in \L$, $\alpha$ can be written as $\gamma\wedge w^{-1}$ for some $\gamma \in \bigwedge^{d-1}\R^n$. 
Then $\alpha' = \diag_{w}(\alpha)$ equals $\diag_w(\gamma)\wedge \1$ and belongs to $\H_\B$. 

Now let $\beta \in \bigwedge^{n-d}\R^n$ and let $\beta'$ denote the image of $\beta$ under diagonal action by $w^{-1}$, 
 $\beta' = {\rm diag}_{w^{-1}}(\beta)$. 
 We claim $\varphi_{\alpha'}(\beta')$ is a scalar multiple of $\varphi_{\alpha}(\beta)$, 
 specifically, for $p,q\in \H_\B$, 
\begin{align*}
\varphi_{\alpha'}(\beta')(p,q)  
&  \  = \  \sum_{I\in \B} \frac{(e_I\wedge \beta')}{\alpha_I'} p_Iq_I  
\ = \  \sum_{I\in \B} \frac{\prod_{j\in [n]\backslash I }w_j^{-1} (e_I\wedge \beta)}{\prod_{i\in I}w_i \alpha_I} p_Iq_I \\
&  \ = \ \prod_{i\in [n]}w_i^{-1}  \cdot  \sum_{I\in \B} \frac{(e_I\wedge \beta)}{\alpha_I} p_Iq_I 
\ = \ \prod_{i\in [n]}w_i^{-1}  \cdot  \varphi_{\alpha}(\beta)(p,q).
 \end{align*}
Now suppose $\beta = \delta \wedge w$ for some $\delta \in \bigwedge^{n-d-1} \R^n$
and let $\delta' $ denote $ {\rm diag}_{w^{-1}}(\delta)$.
Then $\beta' $ equals $ {\rm diag}_{w^{-1}}(\delta)\wedge {\rm diag}_{w^{-1}}(w) = \delta'\wedge \1$.
The equation above then implies that $\varphi_{\alpha}( \delta \wedge w)$ is a scalar multiple of $\varphi_{\alpha'}( \delta'\wedge \1)$. 
By Lemma~\ref{lem:wedge1}, $\alpha' ={\rm diag}_{w}(\alpha)$ belongs to its kernel. 
\end{proof}

To understand the dimension of $\H_\B$ and the behavior of the map $\varphi_\alpha$, we need to introduce
some matroid theory. 
The matroid associated to $\L\in \Gr(d,n)$ is the matroid of rank-$d$ on $n$ elements represented by the restriction of 
the linear forms $e_1^{\vee}, \hdots, e_n^{\vee}\in (\R^n)^{\vee}$ to $\L$.  If we write $\L$ as the rowspan of a $d\times n$ matrix, then 
this is the matroid on its $n$ columns.  
In \cite{PS}, Proudfoot and Speyer show that there is a flat degeneration of the homogeneous coordinate ring of $\L^{-1}$ to the Stanley-Resiner ideal 
of \emph{broken circuit complex} of this matroid.

\begin{Def} Given a matroid $M$ on $n$ elements, we fix an ordering $1<2<\hdots<n$ on the elements of $M$. 
A \textbf{circuit} is a minimally dependent subset of $M$. A circuit of size one is called a \textbf{loop} and a matroid with 
no loops is called \textbf{loop-less}.   A \textbf{broken circuit} of $M$ is a circuit of $M$ with its maximal element removed. 
The \textbf{broken-circuit complex} ${\rm BCC}(M)$ is the simplicial complex on the elements $\{1,\hdots, n\}$ whose
faces are subsets $S$ that contain no broken circuit, \cite{WhiteMatroids}. This simplicial complex has dimension ${\rm rank}(M)$ and the number of its
facets is a matroid invariant, which gives the degree of $\L^{-1}$ for any linear space $\L$ with corresponding matroid $M$ \cite{PS}. 
\end{Def}

 \begin{lem}\label{lem:dimHB}
 Fix a loop-less matroid on $\{1,\hdots, n\}$ of rank $d$ with bases $\B$,
 and let $\F \subset \B$ denote the facets of the broken circuit complex.  For bases $B, B'\in \B$, 
say that $B< B'$ if $|B\cup B'| = d+1$ and $B'\backslash B$ is the maximum element of the unique circuit contained in $B\cup B'$. 
This can be extended by transitivity to a partial order on $\B$. Furthermore, $\mathcal{F}$ is the set of maximal elements of this partial order. 
 \end{lem}
\begin{proof}
To show that we can extend this to a partial order on $\B$, we need to show that there are no cycles 
$B_1< B_2 < \hdots < B_r < B_1$. Suppose $|B\cup B'| = d+1$, meaning that for some $i,j\in [n]$, $B'\backslash B = \{i\}$ and 
$B\backslash B' = \{j\}$. Let  $C$ be the unique 
circuit in $B\cup B'$.  Since neither $B$ nor $B'$ contains $C$, both $i$ and $j$ are elements of $C$. In particular, if $i=\max(C)$, then $j<i$. 
Thus, if $B<B'$, then $\sum_{b\in B} b < \sum_{b'\in B'} b'$. This rules out the possibility of any cycle. 

If $B \in \B\backslash \mathcal{F}$, then $B$ contains some broken circuit $C\backslash \max(C)$. 
Since $|C|\neq 1$, the set $B\cup C$ contains some other basis $B'\neq B$. Then $B<B'$ and $B$ is not maximal. 
Conversely, if $B$ is not maximal, then there is some basis $B'$ and circuit $C\subset B\cup B'$ with $B'\backslash B = \{\max(C)\}$. 
Then $B$ contains the broken circuit $C\backslash \max(C)$ and thus is not in $\mathcal{F}$. 
\end{proof}

\begin{lem}\label{claim:dimH1}
Let $\B$ be the set of bases of a loop-less matroid, and let 
$\F \subset \B$ denote the facets of the corresponding broken circuit complex.
The linear space $\H_\B \subset \bigwedge^d\R^n$ defined in \eqref{eq:HB} has dimension at most $|\F|$. 
\end{lem}

\begin{proof}
In fact, we will show that the dual vectorspace $\H_\B^{\vee}$ has dimension at most $|\F|$. 
As done above, we consider the restriction map $\rb:(\bigwedge^d\R^n)^\vee \rightarrow \H_\B^{\vee}$ of linear forms 
on $\bigwedge^d\R^n$ to linear forms on $\H_\B^\vee$  and identify $(\bigwedge^d\R^n)^\vee$ with $\bigwedge^{n-d}\R^n$.
That is, for $\alpha \in \H_\B$ and $\beta \in \bigwedge^{n-d}\R^n$, $\rb(\beta)(\alpha) = \alpha\wedge \beta$. 
Since $\H_\B$ is contained in ${\rm span}\{\gamma \wedge \1: \gamma\in \bigwedge^{d-1}\R^n\}$, the kernel of the 
restriction map $\rb$ contains all elements of the form $\delta \wedge \1$ where $\delta \in \bigwedge^{n-d-1}\R^n$. 
We will use this to show that for any basis $B\in \B\backslash \mathcal{F}$,
\[\rb (e_{[n]\backslash B}) \ \ \in  \ \ {\rm span}\{ \rb(e_{[n]\backslash B'}) \ : \ B < B'\} .\]
Take any basis $B\in \B\backslash \mathcal{F}$. By definition, $B$ contains some broken circuit 
$C'\backslash \max(C')$. Let $C $ denote the union $ B\cup C'  =  B\cup \max(C')$. 
Since $e_{[n]\backslash C}$ belongs to $\bigwedge^{n-d-1}\R^n$, the element $ e_{[n]\backslash C} \wedge \1$ belongs to the 
kernel of $\rb$. Noting that $e_{[n]\backslash C} \wedge e_j = 0$ for all $j\not\in C$ then gives that 
\begin{equation}\label{eq:Pi}
0  \ \ = \  \ \rb(e_{[n]\backslash C} \wedge \1)  \ = \  \sum_{i\in C } \rb(e_{[n]\backslash C} \wedge e_i ).
 \end{equation}
 Note that, up to sign, $e_{[n]\backslash C} \wedge e_i$ equals $e_{[n]\backslash I}$ where $I = C\backslash i$. Recall that 
 $\H_\B$ belongs to the span of $\{e_I : I\in \B\}$. Thus if $I$ is not a basis, then  the restriction of $e_{[n]\backslash I} $ to $\H_\B$
is identically zero.  In particular, the sum in \eqref{eq:Pi} can be taken over $\{i\in C : C\backslash i \in \B\}$. 

If $I =C \backslash i$ is a basis, then either $i = \max(C')$ and $I = B$ or  $\max(C') \in I$ and $I\backslash B = \max(C')$ is the 
maximal element of the unique circuit $C'$ contained in $I\cup B = C$. 
Thus \eqref{eq:Pi} writes $\rb (e_{[n]\backslash B})$ as a linear combination of $\{\rb(e_{[n]\backslash B'}): B'\in \B, \ B<B'\}$.  
Since the maximal elements of this partial order are exactly the facets $\mathcal{F}$ of the broken circuit complex, this implies that 
$\H_{\B}^\vee  \subseteq  {\rm span}\{ \rb(e_{[n]\backslash F})  :  F \in \F\}$. 
In particular,  $\dim(\H_{\B})=\dim(\H_{\B}^{\vee})\leq |\mathcal{F}|$.
\end{proof}

Now we are ready to prove Theorem~\ref{thm:detrep}.

\begin{proof}[Proof of Theorem~\ref{thm:detrep}]\label{proof:detrep}
Fix a $(d-1)$-dimensional real linear space $\L\subseteq \P^{n-1}(\C)$ not contained in any coordinate hyperplane with Pl\"ucker coordinates $[\alpha] = p(\L)$. The real torus points $(\R^*)^n\cap \L^{-1}$ are Zariski-dense in $\L^{-1}$. 
So consider a point $w \in (\R^*)^{n}$. 
By Lemma~\ref{lem:wedge2}, if $[w] \in \L^{-1}$, then $\diag_w(\alpha)$ belongs to the kernel of the quadratic form $\varphi_\alpha \wedge w \wedge \delta$ on $\H_\B$ for all $\delta\in\bigwedge^{n-d-1}\R^n$.
In particular, this kernel only depends on the point $w$. Since $\varphi_{\alpha}$ is non-degenerate, it follows that $\varphi_{\alpha}$ is a Liv{\v{s}}ic-type determinantal representation for some projective variety $X$ containing $\L^{-1}$. Note that $\dim X=\dim \L^{-1}$, cf. \cite[Cor. 2.4]{SV}.  Then by \cite[Cor. 2.13]{SV}, $k= \dim(\H_\B) \geq \deg(X) \geq \deg(\L^{-1})$. By the results of \cite{PS}, 
$\deg(\L^{-1})$ equals the number $|\F|$ of facets of the broken circuit complex.
Putting this together with Lemma~\ref{claim:dimH1}, shows that $k = \dim(\H_\B) = |\F|$ and $X = \L^{-1}$. 
Therefore $\varphi_\alpha$ is a symmetric Liv{\v{s}}ic-type determinantal representation of $\L^{-1}$.
As discussed after its definition, $\varphi_\alpha$ is also positive definite at $\L^{\perp}$
\end{proof}

As in Proposition~\ref{prop:livsicchow}, this gives a determinantal representation of the Chow form of $\L^{-1}$.

\begin{cor}\label{cor:detVectors}
Let $\L \subseteq \P^{n-1}(\C)$ be a real linear subspace of dimension $d-1$ not contained
 in any coordinate hyperplane. Let $[\alpha] = p(\L) \in \P(\bigwedge^d \R^n)$ be the vector of 
 Pl\"ucker coordinates of $\L$ and $\B={\rm supp}(\alpha)$. 
Then there exist vectors $v_{I}\in \R^{\deg(\L^{-1})}$ for each $I\in \B$ 
so that $\M\in \Gr(n-d,n)$ intersects $\L^{-1}$ if and only if the determinant of the matrix
\[ \sum_{I\in \B} \frac{p_I(\M^{\perp})}{p_I(\L)} v_Iv_I^T \ \ = \ \  \sum_{I\in \B} \frac{(-1)^{s(I)} \cdot p_{[n]\backslash I}(\M)}{p_I(\L)} v_Iv_I^T   \]
is zero. This determinant is the Chow form of $\L^{-1}$.
\end{cor}

\begin{proof}
By the proof of Theorem~\ref{thm:detrep}, $\H_{\B}$ has dimension $|\F| = \deg(\L^{-1})$ and the vectors 
$\{\rb(e_{F}^*) : F\in \F\}$ form a basis for $\H_{\B}^{\vee}$. 
Thus for each $I\in \B$, we can write $\rb(e_{I}^*)$ as 
\[ \rb(e_{I}^*)   \ \ = \ \ \sum_{F\in \F} v_{I,F} \cdot  \rb(e_{F}^*) \ \ \text{ for some vectors } \ \ v_I\in \R^{|\F|}.\]
In particular for $F\in \F$, the vectors $v_F$ are unit coordinate vectors. 

Thus for any $\alpha =  \sum_{I\in \B} \alpha_I e_I \in \bigwedge^d\R^n$ and 
$\gamma = \sum_{I\in \binom{[n]}{d}} \gamma_I e_I^* \in (\bigwedge^{d}\R^n)^{\vee}$, we have that 
\begin{equation}\label{eq:matrixVarPhi}
 \varphi_{\alpha}(\gamma) \ \ = \ \ 
\sum_{I \in \B} \frac{\gamma_I}{\alpha_I}\cdot \sum_{F, F' \in \F} v_{I,F}\cdot v_{I,F'} \cdot  \rb(e_{F}^*) \otimes \rb(e_{F'}^*). 
\end{equation}
Taking the basis of $\H_{\B}$ dual to the basis $\{\rb(e^*_F):F\in \F\}$ of $\H_\B^\vee$
 represents the symmetric bilinear form 
$\varphi_{\alpha}(\gamma)$ as the $|\F|\times|\F|$ real symmetric matrix $\sum_{I\in \B} (\gamma_{I}/\alpha_I) \cdot v_Iv_I^T$.
Now let $\beta$ denote the point in $\bigwedge^{n-d}\R^n$ identified with $\gamma\in (\bigwedge^{d}\R^n)^{\vee}$.
That is $\beta = \sum_I \beta_{I} e_{[n]\backslash I}$, where $\beta_I = (-1)^{s(I)} \gamma_I$ for all $I\in \binom{[n]}{d}$. 
Then  $\varphi_\alpha(\beta) $ is represented by  the matrix
\[\sum_{I\in \B} \frac{(-1)^{s(I)}\beta_{I}}{\alpha_I} \cdot v_Iv_I^T  \ \ = \ \ \sum_{I\in \B} \frac{\gamma_I}{\alpha_I} \cdot v_Iv_I^T. \]
By Theorem~\ref{thm:detrep} and Proposition~\ref{prop:livsicchow}, this matrix is a determinantal representation of the
Chow form of $\L^{-1}$. In particular, if $[\beta] = p(\M) \in \bigwedge^{n-d}\R^n$ are the Pl\"ucker coordinates of some $\M \in \Gr(n-d,n)$,
then  $\L^{-1}$ intersects $\M$ if and only if the determinant of this matrix equals zero. 
Note that if $[\beta] = p(\M)$, then $[\gamma] = p(\M^{\perp})$, where $\M^{\perp}\in \Gr(d,n)$ is the orthogonal complement of $\M$.  
\end{proof}

\begin{example}
Consider the rank-3 matroid on $5$ elements with circuits $\{124, 135, 2345\}$ and let $\B$ denote its set of bases.
For example, this matroid is represented by the linear space
\[\L  \ \ = \ \ \text{rowspan}\begin{pmatrix} 
1 & 0 & 0 & 1& 1 \\
0 & 1 & 0 & 1& 0 \\
0 & 0 & 1 & 0& 1 \\
\end{pmatrix}.\]
The broken circuits are $\{12,13,234\}$ and
the broken circuit complex is the simplicial complex with facets $\mathcal{F} = \{145, 235, 245, 345\}$.
The other four bases are $\B\backslash \F = \{123, 125, 134, 234\}$.
As in the proof of Lemma~\ref{claim:dimH1}, every $C\in \binom{[5]}{4}$ gives a linear relation on $\{\rb(e_{I}^*) : I \in \B\}$ as in \eqref{eq:Pi}. 
For example, taking $C = \{1,2,3,4\}$ and $[5]\backslash C = \{5\}$ gives the relation
\[0\ =\  \sum_{i=1}^4\rb(e_{5}\wedge e_i) \ = \ -\sum_{i=1 }^4  \rb(e_{\{i,5\}}) \ 
=  \ -\rb(e^*_{\{123\}})-\rb(e^*_{\{134\}})+\rb(e^*_{\{234\}}),\]
since $\rb(e^*_{\{124\}})=0$. Ranging over all $C\in \binom{[5]}{4}$, we get five linear relations, any four of which are linearly independent.
Solving these linear equations we can write $\rb(e_{I}^*)$ as $\sum_{F\in \F} v_{I,F} \cdot \rb(e_{F}^*)$ for some $v_{I}\in \R^{|\F|} = \R^4$.
In this example, we have that 
\[ 
\begin{pmatrix}  \ \\ v_{I} \\ \ \end{pmatrix}_{I\in \B} \ \ = \ \
\bordermatrix{
F \ \backslash  \ I	  &145 &  235 	&  245 & 345 &  123 & 125 &  134 &  234\cr
145 & 1 & 0 & 0 & 0 & 1 & 1 & -1 & 0 \cr
235 & 0 & 1 & 0 & 0 & 1 & 0 & 0 & 1 \cr
245  & 0 & 0 & 1 & 0 & -1 & -1 & 0 & -1 \cr
345  & 0 & 0 & 0 & 1 & 0 & 0 & 1 & 1 \cr
}.
\]
Suppose $\alpha \in \bigwedge^3 \R^5$ with ${\rm supp}(\alpha) = \B$. For any $\beta\in \bigwedge^2\R^5$, the symmetric bilinear form $\varphi_{\alpha}(\beta)$ 
can be represented by the $4\times 4$ symmetric matrix from equation \eqref{eq:matrixVarPhi}: 
\[\begin{pmatrix}  
 \frac{\beta_{45}}{\alpha_{123}}+ \frac{\beta_{34}}{\alpha_{125}}+ \frac{\beta_{25}}{\alpha_{134}}+\frac{\beta_{23}}{\alpha_{145}} & \frac{\beta_{45}}{\alpha_{123}} &
   -\frac{\beta_{45}}{\alpha_{123}}-\frac{\beta_{34}}{\alpha_{125}} & -\frac{\beta_{25}}{\alpha_{134}} \medskip \\
 \frac{\beta_{45}}{\alpha_{123}} & \frac{\beta_{45}}{\alpha_{123}}-\frac{\beta_{15}}{\alpha_{234}}+\frac{\beta_{14}}{\alpha_{235}} & -\frac{\beta_{45}}{\alpha_{123}}+\frac{\beta_{15}}{\alpha_{234}} &
  - \frac{\beta_{15}}{\alpha_{234}}  \medskip\\
 -\frac{\beta_{45}}{\alpha_{123}}-\frac{\beta_{34}}{\alpha_{125}} & -\frac{\beta_{45}}{\alpha_{123}}+\frac{\beta_{15}}{\alpha_{234}} &
   \frac{\beta_{45}}{\alpha_{123}}+\frac{\beta_{34}}{\alpha_{125}}-\frac{\beta_{15}}{\alpha_{234}}-\frac{\beta_{13}}{\alpha_{245}} & \frac{\beta_{15}}{\alpha_{234}}  \medskip\\
 -\frac{\beta_{25}}{\alpha_{134}} & -\frac{\beta_{15}}{\alpha_{234}} & \frac{\beta_{15}}{\alpha_{234}} & \frac{\beta_{25}}{\alpha_{134}}-\frac{\beta_{15}}{\alpha_{234}}+\frac{\beta_{12}}{\alpha_{345}} 
  \end{pmatrix}. \]
  In particular, if $[\alpha] = p(\L)$ and $[\beta] = p(\M)$ for some $\M\in \Gr(2,5)$, 
  then the determinant of this matrix vanishes if and only if the intersection $\M \cap \L^{-1}$ is non-empty. The determinant is the Chow form of the variety $\L^{-1}$ in $\R[\beta_{ij} : 1\leq i < j \leq 5]$. \hfill $\diamond$
\end{example}

\section{Explicit formulas in the uniform case}\label{sec:uniform}
For generic linear spaces $\L\in {\rm Gr}(d,n)$, all the Pl\"ucker coordinates $p_I(\L)$ are non-zero.  
In this case, we can give explicit formulas for the vectors in our determinantal representation 
of the Chow form of $\L^{-1}$ and for the monomial expansion of this determinant. 

\begin{thm}\label{thm:uniformVec}
Let  $\L \subseteq \P^{n-1}(\C)$ be a real linear subspace of dimension $d-1$ with Pl\"ucker coordinates $[\alpha] = p(\L)$, 
and suppose that ${\rm supp}(\alpha) = \binom{[n]}{d}$. 
Then $k = \deg(\L^{-1})= \binom{n-1}{d-1}$ and we can identify $\R^k$ with ${\rm span} \{e_{K} : K\in \binom{[n-1]}{d-1}\}$. For $I\in\binom{[n]}{d}$, define the vector $v_I \in \R^k$ by 
\begin{equation}\label{eq:uniformVec}
v_{I}  \ \  = \ \  
\begin{cases}
e_{I\backslash \{n\} } &\text{if } n\in I\\
\sum_{i\in I}(-1)^{s(i, I )}  e_{I\backslash\{i\}} & \text{if } n \not\in I,
\end{cases}
\end{equation}
where $i$ is the $s(i,I)$th element of $I$.
Then $\varphi_{\alpha} =  \sum_{I\in\binom{[n]}{d}} \alpha_I^{-1} e_I  \cdot v_Iv_I^T \in \bigwedge^d\R^n \otimes \Sym_k(\R) $
is a Liv{\v{s}}ic-type determinantal representation of $\L^{-1}$. 
With respect to $\M\in \Gr(n-d,n)$, the Chow form of $\L^{-1}$ is the determinant of the $k\times k$ matrix
\begin{equation}\label{eq:uniformMatrix}
 \sum_{I\in \binom{[n]}{d}} \frac{p_I(\M^{\perp})}{p_I(\L)} v_Iv_I^T \ \ = \ \  \sum_{I\in \binom{[n]}{d}} \frac{(-1)^{s(I)} \cdot p_{[n]\backslash I}(\M)}{p_I(\L)} v_Iv_I^T.   
 \end{equation}
\end{thm}
\begin{proof}
Consider the uniform matroid with bases $\B = \binom{[n]}{d}$ and circuits $\binom{[n]}{d+1}$. 
Its broken circuits are $\binom{[n-1]}{d}$ and the facets of the broken circuit complex 
are  $\F = \{F\in \B : n\in F\}$.
By the proof of Theorem~\ref{thm:detrep},
the vectors $\{\rb(e_{F}^*) : F\in \F\}$ form a basis for $\H_{\B}^{\vee}$. 
The dual basis of $\H_\B$ equals $\{e_K \wedge \1: K\in \binom{[n-1]}{d-1} \}$.
To see this, note that for $F\in \F$ and $K\in \binom{[n-1]}{d-1}$, 
\[ \rb(e_{F}^*)(e_K\wedge \1)  \ \ = \ \ (-1)^{s(F)} \cdot  e_K \wedge  \1  \wedge e_{[n]\backslash F}  \ \ =  \ \ \begin{cases}
1 & \text{if } K\cup\{n\} = F\\
0 & \text{otherwise.} \end{cases}
\]
In particular, for $I\in \B\backslash \F$, we can write the linear forms $\rb(e_{I}^*) = \sum_{F\in \F}  v_{I,F} \cdot  \rb(e_{F}^*) $ where 
\[
v_{I,F}  \ \ = \ \ \rb(e_{I}^*)(e_{F\backslash \{n\}}\wedge \1)  \ \ 
=  \ \ \begin{cases}
(-1)^{s(i,I)+d-1} & \text{if } I =  F \cup\{i\} \backslash \{n\} \text{ for some $i\in [n-1]$}\\
0 & \text{otherwise.} \end{cases}
\]
Multiplying $v_I$ by $(-1)^{d-1}$ does not change the matrix $v_Iv_I^T$. 
Therefore the vectors $v_I$ are exactly those described in the proof of Corollary~\ref{cor:detVectors}. \end{proof}

Since all of the coefficients matrices $v_Iv_I^T$ in the representation \eqref{eq:uniformMatrix} have rank-one, 
the monomial expansion of the determinant involves only square-free monomials. By the Cauchy--Binet Theorem, 
the coefficients of these monomials are the squares of the corresponding maximal minors of the $\binom{n-1}{d-1}\times \binom{n}{d}$ matrix $(v_I : I\in \binom{[n]}{d})$. 

The vectors $\{v_I:I\in \binom{[n]}{d} \}$ appearing in Theorem~\ref{thm:uniformVec} have special structure, 
which we can exploit to get a more explicit formula for the Chow form of $\L^{-1}$. 
Specifically, they appear in the boundary
operator of the ``complete'' simplicial complex of dimension $(d-1)$ on $n$ vertices. 
For $d =2$, this is just the complete graph on $n$ vertices.

\begin{example}\label{ex:Gr24}
Consider the uniform matroid of rank 2 on $\{1,2,3,4\}$. 
Using Theorem~\ref{thm:uniformVec}, we define vectors $v_{i4} = e_{i}$ for $i=1, 2,3$ and $v_{ij}=  e_{i}- e_{j}$ for $1\leq i< j \leq 3$.
The linear space $\H_{\binom{[4]}{2}}$ is spanned by the rows of the matrix $(v_{ij}: 1\leq i < j \leq 4)$:
\[ \begin{pmatrix}v_{14} & v_{24} & v_{34} & v_{12} & v_{13} & v_{23} \end{pmatrix} \ \  =  \ \
\begin{pmatrix} 
1 & 0 & 0 & 1& 1& 0 \\
0 & 1 & 0 & -1 & 0 & 1\\
0 & 0 & 1 & 0 & -1 & -1
\end{pmatrix}.\]
The rows of this matrix are the coordinates of $e_1\wedge \1 $, $e_2\wedge \1 $, and $e_3\wedge \1 $. 
Theorem~\ref{thm:uniformVec}   states that for a linear space $\L\in \Gr(2,4)$ with non-zero Pl\"ucker coordinates $\alpha_{ij}=p_{ij}(\L)$, 
the Chow form of $\L^{-1}$ in the variables $\gamma_{ij} =p_{ij}(\M^{\perp})$ equals the determinant 
\[ \det
\begin{pmatrix}
 \frac{\gamma_{12}}{\alpha_{12}}+\frac{\gamma_{13}}{\alpha_{13}}+\frac{\gamma_{14}}{\alpha_{14}} & -\gamma_{12}/\alpha_{12} &-\gamma_{13}/\alpha_{13} \\
 -\gamma_{12}/\alpha_{12} & \frac{\gamma_{12}}{\alpha_{12}}+\frac{\gamma_{23}}{\alpha_{23}}+\frac{\gamma_{24}}{\alpha_{24}} & -\gamma_{23}/\alpha_{23} \\
 -\gamma_{13}/\alpha_{13} & -\gamma_{23}/\alpha_{23} & \frac{\gamma_{13}}{\alpha_{13}}+\frac{\gamma_{23}}{\alpha_{23}}+\frac{\gamma_{34}}{\alpha_{34}} \\
\end{pmatrix}.
 \]
 Clearing denominators by multiplying by $\prod_{i,j}\alpha_{ij}$ give the bihomogeneous equation \smallskip \\
 \begin{tabular}{l}
$ \alpha_{23} \alpha_{24}  \alpha_{34}  \gamma_{12} \gamma_{13} \gamma_{14}  + \alpha_{13} \alpha_{24}  \alpha_{34}  \gamma_{12} \gamma_{14}  \gamma_{23} + 
 \alpha_{12} \alpha_{24}  \alpha_{34}  \gamma_{13} \gamma_{14}  \gamma_{23} + \alpha_{14}  \alpha_{23} \alpha_{34}  \gamma_{12} \gamma_{13} \gamma_{24}$\\  
 $+ \alpha_{12} \alpha_{23} \alpha_{34}  \gamma_{13} \gamma_{14}  \gamma_{24}  + \alpha_{13} \alpha_{14}  \alpha_{34}  \gamma_{12} \gamma_{23} \gamma_{24}  + 
 \alpha_{12} \alpha_{14}  \alpha_{34}  \gamma_{13} \gamma_{23} \gamma_{24}  + \alpha_{12} \alpha_{13} \alpha_{34}  \gamma_{14}  \gamma_{23} \gamma_{24} $\\ 
 $+ \alpha_{14}  \alpha_{23} \alpha_{24}  \gamma_{12} \gamma_{13} \gamma_{34}  + \alpha_{13} \alpha_{23} \alpha_{24}  \gamma_{12} \gamma_{14}  \gamma_{34}  + 
 \alpha_{13} \alpha_{14}  \alpha_{24}  \gamma_{12} \gamma_{23} \gamma_{34}  + \alpha_{12} \alpha_{14}  \alpha_{24}  \gamma_{13} \gamma_{23} \gamma_{34}$ \\ 
 $+ \alpha_{12} \alpha_{13} \alpha_{24}  \gamma_{14}  \gamma_{23} \gamma_{34}  + \alpha_{13} \alpha_{14}  \alpha_{23} \gamma_{12} \gamma_{24}  \gamma_{34}  + 
 \alpha_{12} \alpha_{14}  \alpha_{23} \gamma_{13} \gamma_{24}  \gamma_{34}  + \alpha_{12} \alpha_{13} \alpha_{23} \gamma_{14}  \gamma_{24}  \gamma_{34}$\\
\end{tabular} \smallskip \\
which vanishes on $[\alpha]=p(\L)$ and $[\gamma]=p(\M^{\perp})$ when $\L^{-1}\cap \M\neq \emptyset$.
 \hfill $\diamond$
\end{example}

In the uniform case with $d=2$, the vectors \eqref{eq:uniformVec} represent the graphic matroid of the complete graph on $n$ vertices. 
The maximal minors of the matrix  $(v_{ij}: 1\leq i < j \leq n)$ are $0, \pm1$, and are zero precisely when the corresponding 
subgraphs contain cycles. See e.g. \cite[Ch. 5]{Oxley}. The matrix in Corollary~\ref{cor:detVectors} is a weighted Laplacian of the complete graph.

\begin{cor}\label{cor:graphic}
Let $\mathcal{T}_n$ denote the set of spanning trees on $n$ vertices. 
If $\L\in \Gr(2,n)$ has no  zero Pl\"ucker coordinates, then
the Chow form of $\L^{\inv}$ in $\C[p_{ij}(\M^{\perp}): 1\leq i<j\leq n]$ is 
\begin{equation}\label{eq:GraphFormula}
  \sum_{T \in \mathcal{T}_n} \prod_{\{i,j\} \in T}  p_{ij}(\M^{\perp}) \cdot \prod_{\{k,\ell\} \not\in T}p_{k \ell}(\L).    
  \end{equation}
\end{cor}

This formula for the Chow form of a reciprocal line has a nice interpretation in terms of the resultant 
of binary forms.  This comes from the observation that the reciprocal linear space of $\L \in \Gr(2,n)$ 
with non-zero Pl\"ucker coordinates is the rational normal curve of degree $n-1$ in $\P^{n-1}(\C)$.  To see this, choose a parametrization 
$[s:t] \mapsto [\ell_1(s,t): \hdots:\ell_n(s,t)]$ of $\L$ where $\ell_1, \hdots, \ell_n$ are linear forms in $\R[s,t]$. 
Then $\L^{-1}$ is the image of $\P^1$ under the rational map $[s:t] \dashmapsto [\ell_1^{-1}: \hdots:\ell_n^{-1} ]$.
Clearing denominators writes $ \L^{-1}$  as the image of $\P^1$  under 
\[[s:t] \dashmapsto \biggl[ \ \prod_{j\neq 1} \ell_j : \ \hdots \ :\prod_{j\neq n} \ell_j  \ \biggl].\]
Since the Pl\"ucker coordinates of $\L$ are non-zero, all the roots in $\P^1$ of $\ell_1, \hdots, \ell_n$  are distinct.  
Up to scaling, the polynomials $\{\prod_{j\neq i} \ell_j \}_{i\in [n]}$ are the interpolators for the roots of $\prod_j \ell_j$, 
and therefore form a basis of the space $\R[s,t]_{n-1}$ of binary forms of degree $n-1$. 
Intersection points of a hyperplane $c^{\perp} = \V(c_1x_1 + \hdots c_nx_n)$ with $\L^{-1}$ 
correspond to roots of the binary polynomial $c_1\prod_{j\neq 1} \ell_j + \hdots +c_n\prod_{j\neq n} \ell_j$.
In particular, a linear space $\M = c_1^{\perp}\cap c_2^{\perp}$ of codimension two intersects $\L^{-1}$ if and only if 
the two binary forms corresponding to $c_1$ and $c_2$ have a common root.  The polynomial in the coefficients of 
binary forms that vanishes when the two forms have a common root is called the \emph{resultant}.
Putting this all together the following interesting consequence of Corollary~\ref{cor:graphic}. 

\begin{cor}\label{cor:resultant}
Suppose $\ell_1, \hdots, \ell_n \in \R[s,t]$ are linear forms $\ell_j= a_{1j}s + a_{2j} t$ with distinct roots in $\P^1$. 
Any two binary forms in $\R[s,t]$ of degree $n-1$ can be represented as 
\[ c_1(x) = c_{11}\prod_{j\neq 1} \ell_j + \hdots +c_{1n}\prod_{j\neq n} \ell_j  
 \ \ \text{ and } \ \
 c_2(x) = c_{21}\prod_{j\neq 1} \ell_j + \hdots +c_{2n}\prod_{j\neq n} \ell_j 
\]
for some $c_{ij}\in \R$. Their resultant equals 
\begin{equation}\label{eq:GraphFormula2}
  \sum_{T \in \mathcal{T}_n} \prod_{\{i,j\} \in T}  p_{ij}(c) \cdot \prod_{\{k,l\} \not\in T}p_{k l}(a), 
  \end{equation}
  where $\mathcal{T}_n$ denotes the set of spanning trees on $n$ vertices, and $p_{ij}(a)$ and $p_{ij}(c)$ denote the $(i,j)$th minors of the $2\times n$ matrices $a = (a_{ij})_{ij}$ and $c = (c_{ij})_{ij}$, respectively. 
\end{cor}

The vectors in Theorem~\ref{thm:uniformVec} are the columns of a boundary operator on a certain 
simplicial complex.  For $d=2$, this complex is the complete graph on $n$ vertices.  The theory of graphic 
matroids led to the explicit formula Corollary~\ref{cor:graphic} for the monomial expansion of the Chow form of a reciprocal linear space. 
Obtaining explicit formulas when $d>2$ involves generalizations of graphic matroids, called \emph{simplicial matroids}, \cite{BK, CL87}. 
Following \cite{BK} and \cite{Hatcher}, we introduce some notation from algebraic topology and this rapidly developing theory.

Given a simplicial complex $G$, let $G_k$ denote the set of $k$-dimensional faces of $G$.
The $k$-incidence matrix of a simplicial complex $G$ is a matrix $\partial_k^G$ 
whose rows are indexed by the $(k-1)$-faces of $G$, whose columns are indexed by $k$-faces of $G$, 
and whose $(I,J)$-th entry is zero if $I\not\subset J$ and $(-1)^j$ when $J = \{u_0,\hdots, u_k\}$ 
with $u_0<\hdots< u_k$ and  $I = J\backslash \{u_j\}$. 
Analogous to the graphical case, a subset $F\subset G_k$ is called a \emph{forest} of $G$ if 
the columns of $\partial_k^G$ are linearly independent, and a \emph{spanning forest} if they form a basis 
for the column
space of $\partial_k^G$. A subset $R\subset G_{k-1}$
is called a \emph{root} of $G$ if the rows of $\partial_k^G$  corresponding to $G_{k-1}\backslash R$ 
form a basis its rowspace. Furthermore, a pair $(F,R)\in G_k\times G_{k-1}$ is called a \emph{rooted forest} of $G$
if $F$ is a forest in $G$ and $R$ is a root of the subcomplex induced by $F$.

The $k$th homology group is the quotient group 
$H^G_k = {\rm ker}_{\Z}(\partial_k^G) / {\rm Im}_{\Z}(\partial_{k+1}^G)$.
As described in \cite{BK}, a subset $S\subset G_k$ of size equal to the rank of $\partial_k^G$  
is a spanning forest of $G$ if and only if its top homology vanishes, \textit{i.e.} $H_k^S=0$. 
In fact, the determinant of the submatrix of $\partial_k^G$ corresponding to a rooted forest $(F,R)$ 
depends on the relative homology groups of the pair $(F,R)$.  Consider the group homomorphism 
$\Psi_{F,R}: \Z F \rightarrow \Z G_{k-1}/\Z R$ given by $x \mapsto \partial^G_k x + \Z R$. 
Then the relative homology group of $(F,R)$ is defined as 
\begin{equation}\label{eq:relHom}
 H_{k-1}(F,R) = {\rm coker}\Psi_{F,R}: = (\Z G_{k-1}/\Z R)/{\rm Im}(\Psi_{F,R}).
\end{equation}
 By \cite[Lemma 17]{BK}, if $(F,R)$ is a rooted forest of $G$, then up to sign, the determinant 
 of the corresponding submatrix of $\partial_k^G$  equals the size of the relative homology group $|H_{k-1}(F,R)|$.
Taking $k=d-1$ and $G = K^{d-1}_n$, the complete simplicial complex of dimension $d-1$ on $n$ vertices, 
gives an explicit formula for the  Chow form of a generic reciprocal linear space.

\begin{thm}\label{thm:uniform}
If all Pl\"ucker coordinates of $\L\in \Gr(d,n)$ are non-zero, then the Chow form of $\L^{-1}$ 
in $\C[p_I(\M^{\perp}) : I\in \binom{[n]}{d}]$ is 
\[
 \sum_{\substack{F  \text{ is a spanning}\\
                                   \text{forest of $K^{d-1}_n$} }} |H_{d-2}(F,R)|^2 
\cdot \prod_{I\in F} p_I(\M^{\perp}) \cdot  \prod_{I\not\in F} p_I(\L),  \]
where $R = \{I\in \binom{[n]}{d-1} : n \in I\}$ and $H_{d-2}(F,R)$ the relative homology group defined in 
\eqref{eq:relHom}. 
\end{thm}
\begin{proof}
Let $G = K^{d-1}_n$ and $R = \{I\in \binom{[n]}{d-1} : n \in I\} \subset G_{d-2}$. 
Consider the $\binom{n}{d-1}\times\binom{n}{d}$ matrix  $\partial_{d-1}^G$
and let $V$ denote the submatrix corresponding to the rows $G_{d-2}\backslash R = \binom{[n-1]}{d-1}$. 
Up to rescaling by $\pm 1$, 
the vectors $\{v_I:I\in \binom{[n]}{d} \}$ appearing in \eqref{eq:uniformVec} are the columns of the matrix $V$. 
By Corollary~\ref{cor:detVectors} and Theorem~\ref{thm:uniformVec} , 
the Chow form of $\L^{-1}$ is the determinant of the $\binom{n-1}{d-1}\times \binom{n-1}{d-1}$ linear matrix 
$ \sum_{I \in \binom{[n]}{d}} p_I(\M^{\perp})/p_I(\L)\cdot v_Iv_I^T$.  By the Cauchy-Binet Theorem, 
this determinant expands as 
\begin{equation}\label{eq:CBexpand}
\sum_{F \subset \binom{[n]}{d} : |F| = \binom{n-1}{d-1} }  \det(V_F)^2 \cdot  \prod_{I\in F} p_I(\M^{\perp})/p_I(\L), 
\end{equation}
where $V_F$ is the submatrix of $V$ corresponding to columns $F$. Note that rescaling the columns 
of $V$ by $\pm 1$ does not change $\det(V_F)^2$.  
By definition, $F$ is a spanning forest of $G$ if and only if $\det(V_F)\neq 0$, so we can restrict the 
sum to be over spanning forests $F$.  

Up to rescaling, the row of $\partial_{d-1}^G$ corresponding to the subset $J \in \binom{[n]}{d-1}$ 
equals the expansion of $e_J\wedge \1$ in the basis $\{e_I : I \in \binom{n}{d}\}$. 
In particular, the rowspan of the linear space $\H_\B$ of \eqref{eq:HB} contains the rowspan of $\partial_{d-1}^G$. 
As in the proof of Corollary~\ref{cor:detVectors}, the rows of $V$ span $\H_\B$ and thus $\partial_d^G$. 
In particular, $R  = \{I\in \binom{[n]}{d-1} : n \in I\}$ is a root of $G$. 
If $F$ is a spanning forest of $G$, then $R$ is also a root of the subcomplex induced by $F$. By \cite[Lemma 17]{BK}, 
$|\det(V_F)|$ equals $H_{d-2}(F,R)$. Plugging this into \eqref{eq:CBexpand} and multiplying by the
non-zero constant $\prod_I p_I(\L)$ gives the result.  
\end{proof}

\begin{example}
For $d>2$, the matrix $\partial_{d-1}^G$ may not be totally unimodular. 
Consider the following spanning forest of $K^2_6 = \binom{[6]}{3}$:
\[ 
F \ = \ \{\{12 3\}, \{124\}, \{136\}, \{145\}, \{156\}, \{2 35\}, \{246\}, \{2 56\}, \{3 4 5\}, \{3 4 6\}\}
\]
This simplicial complex is double-covered by the icosahedron and forms a 6-vertex triangulation of the real projective plane.  
The corresponding minor $\det(V_F)$ of the $10\times 20$ matrix $V$ is $\pm2$. 
To see this, note that any root of $F$ is a tree $T \subset F \cong\P^2(\R)$. 
The minor $|\det(V_F)|$ equals the size of the relative homology group $H_1(\P^2(\R),T)$. 
Since $T$ is contractible, we have $H_1(\P^2(\R),T) \cong H_1(\P^2(\R))\cong \Z/2\Z$.
In this example, all minors of $V$ belong to $\{0,\pm1, \pm2\}$. \end{example}


\section{The Bi-Chow form and Hadamard products of linear spaces}\label{sec:Hadamard}

In this section we define a more symmetric version of the Chow form of a reciprocal linear space and relate it 
to the Hadamard product of linear spaces, studied in \cite{Hadamard}. 

It is interesting to observe that the condition on linear spaces $\L, \M$ that $\L^{-1}\cap\M \cap(\C^*)^n$ is non-empty is symmetric 
in $\L$ and $\M$. That is, there is a point $w$ in $\L^{-1}\cap\M \cap(\C^*)^n$ if and only if there 
is a point in $\L\cap\M^{-1} \cap(\C^*)^n $, namely $w^{-1}$.  To reflect this symmetry, we consider 
this condition jointly in $(\L,\M)$.  It defines a locus of codimension-one in $\Gr(d,n)\times \Gr(n-d,n)$, on which the following polynomial vanishes. 

\begin{Def} 
We define the \textbf{Bi-Chow form} of reciprocal linear spaces to be 
the element of the coordinate ring of $\Gr(d,n)\times \Gr(n-d,n)$ given by 
\[
P(L,M) \ \ = \ \
 \sum_{\substack{F  \text{ is a spanning}\\
                                   \text{forest of $K^{d-1}_n$} }} c_F
\cdot \prod_{I\in F} p_I(M^{\perp}) \cdot  \prod_{I\not\in F} p_I(L),   \]
where $c_F\in\Z_+$ is the constant  $c_F = |H_{d-2}(F,R)|^2$ from Theorem~\ref{thm:uniform}. We also identify $p_I(M^{\perp})$ with $(-1)^{s(I)} p_{[n]\backslash I}(M)$. 
Up to multiplication by the product of the Pl\"ucker coordinates $\prod_Ip_I(L)$, this equals the determinant of the $\binom{n-1}{d-1} \times \binom{n-1}{d-1}$
matrix in \eqref{eq:uniformMatrix}. 
\end{Def}

This polynomial is bi-homogeneous in the Pl\"ucker coordinates of $L$ and $M$. 
Specializing either coordinate to a specific (generic) linear space 
$\L\in \Gr(d,n)$ or $\M\in \Gr(n-d,n)$ results in the Chow form of reciprocal linear spaces $\L^{-1}$ and $\M^{-1}$, respectively. 
From this we conclude that the Bi-Chow form has degree $\binom{n-1}{d-1}$ in the Pl\"ucker coordinates of $M$ and degree 
$\binom{n-1}{n-d-1}=\binom{n-1}{d}$ in the Pl\"ucker coordinates of $L$.  The total degree, $\binom{n}{d} = \binom{n-1}{d-1}+ \binom{n-1}{d}$, is the 
number of facets in the complete simplicial complex $K^{d-1}_n$. 
Example~\ref{ex:Gr24} gives an explicit formula for the Bi-Chow form of reciprocal linear spaces on $\Gr(2,4)\times \Gr(2,4)$.

In recent work, Bocci, Carlini, and Kileel study Hadamard products of linear spaces \cite{Hadamard}. 
The \textbf{Hadamard product} of two varieties $X,Y\subset \P^{n-1}(\C)$ is 
defined as the image of $X\times Y$ under the coordinate-wise multiplication map $(x,y)\dashmapsto [x_1y_1: \hdots: x_ny_n]$. 
When the Hadamard product of two generic linear spaces forms a hypersurface in $\P^{n-1}(\C)$, 
it closely relates to the corresponding Bi-Chow form.  
We can use this to get an equation for $\L\star\M$ for generic linear 
spaces $\L\in \Gr(d,n), \M\in \Gr(n-d,n)$.

\begin{thm}\label{thm:Hadamard}
Let $P(L, M)$ denote the Bi-Chow form on $\Gr(d,n)\times \Gr(n-d,n)$.
For generic $\L\in \Gr(d,n)$, $\M\in \Gr(n-d,n)$, 
the Hadamard product $\L\star \M$ is a hypersurface 
defined by $\prod_{i=1}^n x_i^{\binom{n-2}{d-1}} \cdot P(x^{-1}\star \L, \M)$, which is a polynomial in 
$\C[x_1, \hdots, x_n]$ of degree $\binom{n-2}{d-1}$. 
\end{thm}

First we check that this holds on the torus $(\C^*)^n$. 

\begin{lem}
For generic $\L\in \Gr(d,n)$, $\M\in \Gr(n-d,n)$, the Hadamard product $\L\star \M$ is a 
hypersurface in $(\C^*)^n$ defined by the Laurent polynomial $P(x^{-1}\star \L, \M)$ in $\C[x_1^{\pm 1}, \hdots, x_n^{\pm 1}]$. 
\end{lem}

\begin{proof}
We claim that $\L^{-1}\cap\M\cap(\C^*)^n$ is nonempty if and only if the all-ones vector $\1$ belongs to $\L\star \M$.
Indeed, if $y\in \L^{-1}\cap\M\cap(\C^*)^n$, then $y^{-1}\in \L$ and $y \in \M$. 
This gives that  $\1 = y^{-1}\star y\in \L\star \M$. 
Similarly if $\1\in \L\star \M$, then for some $y\in \M$, $y^{-1}$ belongs to $\L$. 
A point $x\in (\C^*)^n$ belongs to $\L\star \M$ if and only if $\1$ belongs to $x^{-1}\star \L \star \M$. 
This happens if and only if the Laurent polynomial $P( x^{-1}\star \L, \M)$ vanishes. \end{proof}

By clearing denominators in $x$ carefully, we find the polynomial defining $\L\star \M$.  

\begin{proof}[Proof of Theorem~\ref{thm:Hadamard}]
Theorem~\ref{thm:uniform} gives a formula for $P(\L,\M)$ when $\L$ and $\M$ are generic.  
Note that $p_I(x^{-1}\star \L)$ equals $\prod_{i\in I} x_i^{-1}p_I(\L)$. 
Since any spanning forest $F$ of $K^{d-1}_{n}$ has size $\binom{n-1}{d-1}$, we may clear denominators 
by multiplying by that power of all the variables, giving
\begin{equation}\label{eq:Had1}
\prod_{i=1}^n x_i^{\binom{n-1}{d-1}} \cdot P(x^{-1}\star \L, \M)  \ \ = \ \ 
\sum_{\substack{F  \text{ is a spanning}\\
                                   \text{forest of $K^{d-1}_n$} }} c_F 
\cdot \prod_{i\in [n]} x_i^{\deg_F(i)}\cdot  \prod_{I\in F} p_I(\M^{\perp}) \cdot  \prod_{I\not\in F} p_I(\L),
\end{equation}
where the degree $\deg_F(i)$ is the number of maximal faces in $F$ containing the vertex $i$.
Multiplying by such a high power of $\prod_{i=1}^n x_i$ was not necessary to clear denominators.  
For a fixed $i\in [n]$, the minimum of $\deg_F(i)$ over all spanning forests $F$ equals $\binom{n-2}{d-2}$. 
This minimum is achieved by taking $F$ to be a cone over the complete simplicial complex $K_{n-1}^{d-2}$. 
Factoring out this power of $\prod_ix_i$ gives a polynomial with no monomial factors, namely the product of 
$P(x^{-1}\star \L, \M)$ with $\prod_i x_i$ raised to the power $ \binom{n-1}{d-2} = \binom{n-1}{d-1} - \binom{n-2}{d-2}$. 
Since each spanning forrest of $K^{d-1}_n$ has $\binom{n-1}{d-1}$ facets, each containing $d$ vertices, 
the polynomial in \eqref{eq:Had1} has degree $d \binom{n-1}{d-1}$.  Dividing it by $(\prod_i x_i)^{\binom{n-2}{d-2}}$
results in a polynomial of degree $d \binom{n-1}{d-1} - n\binom{n-2}{d-2} = \binom{n-2}{d-1}$, which matches the degree 
for $\L\star \M$ obtained in \cite[Theorem~6.8]{Hadamard}.  
\end{proof}

\begin{example} The Bi-Chow form on $(\L, \M)\in \Gr(2,4)\times \Gr(2,4)$ is given in Example~\ref{ex:Gr24} in coordinates $\alpha_{ij}=p_{ij}(\L)$ and $\gamma_{ij} =p_{ij}(\M^{\perp})$.
Replacing $ \alpha_{ij}=p_{ij}(\L)$ with $ x_i^{-1}x_j^{-1} \alpha_{ij}=p_{ij}(x^{-1}\star\L)$ and clearing denominators gives the polynomial defining $\L\star \M$: 
\begin{center} \begin{tabular}{l}
$ \alpha_{23} \alpha_{24}  \alpha_{34}  \gamma_{12} \gamma_{13} \gamma_{14}\cdot x_1^2  + \alpha_{13} \alpha_{24}  \alpha_{34}  \gamma_{12} \gamma_{14}  \gamma_{23}\cdot x_1x_2 + 
 \alpha_{12} \alpha_{24}  \alpha_{34}  \gamma_{13} \gamma_{14}  \gamma_{23}\cdot x_1x_3 $\\ 
 $+ \alpha_{14}  \alpha_{23} \alpha_{34}  \gamma_{12} \gamma_{13} \gamma_{24} \cdot x_1x_2 + \alpha_{12} \alpha_{23} \alpha_{34}  \gamma_{13} \gamma_{14}  \gamma_{24}\cdot x_1x_4  
 + \alpha_{13} \alpha_{14}  \alpha_{34}  \gamma_{12} \gamma_{23} \gamma_{24}\cdot  x_2^2$\\ 
 $+ \alpha_{12} \alpha_{14}  \alpha_{34}  \gamma_{13} \gamma_{23} \gamma_{24}\cdot x_2 x_3  + \alpha_{12} \alpha_{13} \alpha_{34}  \gamma_{14}  \gamma_{23} \gamma_{24} \cdot x_2x_4
 + \alpha_{14}  \alpha_{23} \alpha_{24}  \gamma_{12} \gamma_{13} \gamma_{34} \cdot x_1x_3$\\  
 $+ \alpha_{13} \alpha_{23} \alpha_{24}  \gamma_{12} \gamma_{14}  \gamma_{34} \cdot x_1x_4  + 
 \alpha_{13} \alpha_{14}  \alpha_{24}  \gamma_{12} \gamma_{23} \gamma_{34}\cdot x_2x_3  + \alpha_{12} \alpha_{14}  \alpha_{24}  \gamma_{13} \gamma_{23} \gamma_{34}\cdot x_3^2$ \\ 
 $+ \alpha_{12} \alpha_{13} \alpha_{24}  \gamma_{14}  \gamma_{23} \gamma_{34}\cdot x_3x_4  + \alpha_{13} \alpha_{14}  \alpha_{23} \gamma_{12} \gamma_{24}  \gamma_{34} \cdot x_2x_4 + 
 \alpha_{12} \alpha_{14}  \alpha_{23} \gamma_{13} \gamma_{24}  \gamma_{34}\cdot  x_3x_4$\\ $+ \alpha_{12} \alpha_{13} \alpha_{23} \gamma_{14}  \gamma_{24}  \gamma_{34}\cdot x_4^2$.\\
\end{tabular}\end{center}
This quadratic was computed in different coordinates in \cite[Example~6.11]{Hadamard}. Every term corresponds to a spanning tree on the complete graph $K_4$, whose edges appear as 
the indices of $\gamma$.  The degree of $x_i$ in this term is one less than the degree of vertex $i$ in the tree.  
 \hfill $\diamond$
\end{example}


\section{Ulrich modules and the entropic discriminant}\label{sec:Ulrich}

Families of polynomial equations all of whose solutions are real often have 
nonnegative discriminants. Given the very special structure of polynomials appearing as discriminants, 
it is natural to ask if these nonnegative discriminants can be written as a sum of squares. 
A beautiful and classical example is the discriminant of the eigenvalues of a symmetric matrix, 
which was shown to be a sum of squares in the entries of the matrix \cite{Bor1846, Ily92, Lax98, New72}. 

Let $X\subset \P^{n-1}(\C)$ be a real variety  of dimension $d-1$ and degree $k$. 
For any linear space $L\in \Gr(n-d,n)$ that does not intersect $X$, we can consider the projection $X\rightarrow \P^{d-1}(\C)$ 
given by projection with center $L$. 
Outside of the branch locus of this map, every point $y$ in $\P^{d-1}(\C)$ has $k$ preimages. 
The discriminant of this projection is a polynomial in $\R[y_1, \hdots, y_d]$ defining this branch locus in the case when it is of pure codimension one.
If the discriminant is square free and nonnegative, the set of its real zeros has codimension at least two. The complement of the real branch locus will then be connected (in the Euclidean topology on $\P^{d-1}(\R)$) and every point in the complement will have the same number of real preimages. The latter is the case when $X$ hyperbolic with respect to $L$. 

For $X = \L^{-1}$, this discriminant is called the \emph{entropic discriminant}, studied by \cite{SSV13}.
Specifically, for a linear space $\L\in \Gr(d,n)$, we can write $\L$ as the rowspan of a $d\times n$ matrix $A$. 
Then the map $[x] \mapsto [Ax]$ defines the projection $\P^{n-1}\dashrightarrow \P^{d-1}$ with center $\L^{\perp}$. 
The entropic discriminant is a polynomial in $\R[y_1, \hdots, y_d]$ vanishing on points $[y]\in \P^{d-1}(\C)$
for which the intersection of $\L^{-1}$ with $\{x: Ax =y\}$ is singular. In \cite{SSV13}, the authors show that this polynomial 
is nonnegative and that its real variety has codimension two in $\P^{d-1}(\R)$.  In Corollary~\ref{cor:SOS}, we prove their conjecture 
that the entropic discriminant is a sum of squares.  

The proof uses the equivalence, recently developed in \cite{KuSh}, between Liv{\v{s}}ic-type determinantal representations
and certain types of \emph{Ulrich modules}. 
An \textbf{Ulrich module} of a polynomial ring $S = \R[x_1,\ldots,x_n]$ is a 
finitely generated, graded Cohen--Macaulay module of $S$ that is generated in degree zero and
whose minimal number of generators equals its multiplicity, \cite{Ulrich1, Ulrich2}. See \cite{ESW03} for connections with resultants and Chow forms. 

For the rest of the section we fix a real variety $X\subset \P^{n-1}(\C)$ of dimension $d-1$ and degree $k$, with a 
Liv{\v{s}}ic-type determinantal representation $\varphi \in \bigwedge^d\R^n \otimes \Sym_k(\R)$
that is definite at some $L\in \Gr(n-d,n)$. 
Let  $S=\R[x_1,\ldots,x_n]$ and let 
$R=\R[y_1,\ldots,y_d]  \hookrightarrow \R[X]$  be the graded Noether normalization 
corresponding to the projection $X \dashrightarrow \P^{d-1}$ with center $L$.

\begin{thm*}[Theorem 4.5 \cite{KuSh}]
Under the conditions above, there is an Ulrich module $M$ over $S$ that is isomorphic to $R^k$ as $R$-module, and 
whose annihilator is  $\mathcal{I}(X)$. Furthermore, there exists an isomorphism of $S$-modules $\psi: M \to \Hom_R(M,R)$ 
such that the induced $R$-bilinear form on $M \cong_R R^k$ is symmetric and admits an orthonormal basis.
\end{thm*}

This has the following concrete consequence. Since $M \cong R^k$ is an $S$-module, multiplication with any element of $S$, in particular 
$x_1, \hdots, x_n$, defines an $R$-linear map $R^k \rightarrow R^k$, and can therefore be represented by 
$k\times k$ matrices with entries in $R$. Since $S$ is commutative, these matrices commute.  
The theorem above implies that the maps $R^k\rightarrow R^k$ induced by the variables $x_1, \ldots, x_n$ can be 
written as commuting \textit{symmetric} matrices $A_1, \ldots, A_n$ whose entries are linear forms in the variables $y_1, \ldots, y_d \in R$.
If the variety $X$ is arithmetically Cohen--Macaulay (ACM), we can use these matrices to write the discriminant of the projection from $L$ as a sum of squares.

\begin{thm} \label{thm:SOS}
With the notation and the conditions above, if $\R[X]$ is a free $\R[y_1, \hdots, y_d]$-module, then the discriminant of the projection $X\rightarrow \P^{d-1}$ is a sum of squares in $\R[y_1, \hdots, y_d]$. 
\end{thm}

\begin{proof}
Since the annihilator of $M$ is $\mathcal{I}(X)$, we have for all $f \in S$ that $f(A_1, \ldots, A_n)=0$ if and only if
 $f \in \mathcal{I}(X)$. Therefore, it makes sense to evaluate an element from $\R[X]$ at $(A_1, \ldots, A_n)$.
Since the coordinate ring $\R[X]$ is a free $R$-module, we can define the trace map 
$\tr: \R[X] \to R$ which sends an element $f \in \R[X]$
 to the trace of a representing matrix of the $R$-linear map $g \mapsto f \cdot g$.
Note that, for $f \in \R[X]$, we have  \[\tr(f) \ \  = \ \ \tr(f(A_1,\ldots,A_n)) \ \ \in  \ \ \R[y_1, \hdots, y_d], \]
 where the trace of $f(A_1,\ldots,A_n)$ is the usual trace of a matrix.
 This is because the minimal polynomial of $f$ over $R$ is the same as the minimal polynomial of the matrix $f(A_1,\ldots,A_n)$.
 
 Let $f_1, \ldots, f_k \in \R[X]$ be an $R$-basis of $\R[X]$ and let $B_i=f_i(A_1, \ldots, A_n)$. 
 Consider the following $k \times k$ matrix with polynomials from $R$ as entries:
 \begin{equation}\label{eq:traceForm}
 H=(\tr(f_i \cdot f_j) )_{1 \leq i,j \leq k}=(\tr(B_i \cdot B_j) )_{1 \leq i,j \leq k}.
 \end{equation}
 The determinant of $H$ is the discriminant of the projection $X\rightarrow \P^{d-1}$. This follows from the fact that 
 an $\R$-algebra, which is finite dimensional as vector space over $\R$, is reduced  if and only if 
 the trace bilinear form
 on it is nondegenerate, cf., e.g., \cite[Prop. 6.6]{AK70}. The zero set of the discriminant is by definition the set of points whose fiber is not reduced.
 
 Finally, consider the map $B\mapsto \vecz(B)$ that takes a $k\times k$ matrix to the length-$k^2$ vector of its entries. 
For any symmetric matrices $B$ and $C$, the trace $\tr(B\cdot C) = \tr(B\cdot C^T)$ equals
the dot product $\vecz(B)\cdot \vecz(C)$.  Therefore the matrix $H$ can be written as $UU^T$ where $U$ is the $k\times k^2$ matrix
whose rows consist of the vectors $\vecz(B_i)$ for $i=1, \hdots, k$. 
 Since the discriminant is the determinant of $H$, the claim follows from the Theorem of Cauchy--Binet.
\end{proof}

In the case of hyperbolic hypersurfaces, connections between positive definite trace forms and determinantal representations 
were studied in \cite{NPT13}.
Using this theory, we can write the entropic discriminant of \cite{SSV13} as a sum of squares. 

\begin{cor}\label{cor:SOS}
The entropic discriminant is a sum of squares.
\end{cor}

\begin{proof}
Let $\L \subseteq \C^n$ always be a linear subspace of dimension $d$ not contained in any coordinate hyperplane.
By Theorem~\ref{thm:detrep}, $\L^{-1}$ has a symmetric and definite Liv{\v{s}}ic-type determinantal representation, 
whose size equals the degree of $\L^{-1}$.  
 By \cite[Prop. 7]{PS} the coordinate ring $\R[\L^{-1}]$ is a free $R$-module.  The result then follows from Theorem~\ref{thm:SOS}. 
\end{proof}

In the following, we explain how to explicitly construct this sum of squares representation for generic $\L\in \Gr(d,n)$. 
Let $v_1, \hdots, v_d \in \R^n$ be a basis for $\L$ and $v_{d+1}, \hdots, v_n\in \R^n$ be a basis for $\L^{\perp}$. 
To use the construction described in \cite{KuSh}, we need to do a change of basis on $\R^n$ to write the projection from 
$\L^{\perp}$ as a coordinate projection.  Specifically, for each $i=1, \hdots, n$ we take $y_i = \sum_j v_{ij} x_j$. 
This identifies $y_1, \hdots, y_n$ with the the dual basis of $v_1, \hdots, v_n$ in $(\R^n)^{\vee}$. 
In this basis, projection from $\L^{\perp}$ is the coordinate projection onto the first $d$ coordinates.

Let  $\alpha = v_1\wedge \hdots \wedge v_d \in \bigwedge^d\R^n$, and take $\varphi = \varphi_\alpha$ to be the 
definite Liv{\v{s}}ic-type determinantal representation of $\L^{-1}$ given by Theorem~\ref{thm:detrep}. 
By the results described above, there is an Ulrich module $M$ corresponding to $\varphi$, 
and $M$ is a free $R$-module, where $R = \R[y_1, \hdots, y_d]$. 
Following the proof of Theorem~\ref{thm:SOS}, for each $i=1, \hdots, n$, 
there exists a $k\times k$ symmetric matrix $A_i$, whose entries are linear forms in 
$\R[y_1, \hdots, y_d]_1$, that represents the action of $y_i$ on $M \cong R^{k}$. 
For $i=1, \hdots, d$, we see that $A_i$ is just $y_i I$, where $I$ is the $k\times k$ identity matrix. 
The construction of the matrices $A_{d+1}, \hdots, A_n$ involves our constructed determinantal representation of $\L^{-1}$. We follow the proofs of \cite[Theorems 3.3., 4.5]{KuSh}.

Since $\varphi$ is definite at $\L^{\perp}$, we can change on coordinates on $\R^k$  to make 
$\varphi\wedge v_{d+1}\wedge \hdots \wedge v_n$ the identity matrix of size $k$. 
For each $j =d+1, \dots, n$, let $w_j$ denote the wedge product of $v_{d+1}\wedge \hdots \wedge v_{j-1} \wedge v_{j+1} \wedge \hdots \wedge v_n$. 
Then $\varphi\wedge w_j$ belongs to $\bigwedge^{n-1}\R^n \otimes \Sym_k(\R)$, which we can identify with $(\R^n)^{\vee} \otimes \Sym_k(\R)$. 
Since $y_1, \hdots, y_n$ is the dual basis of $v_1, \hdots, v_n$ in $(\R^n)^{\vee}$, then
\[ \varphi\wedge w_j 
\ =  \ \sum_{i=1}^n y_i \cdot \varphi \wedge w_j \wedge v_i
 \ = \ (-1)^{n-j} \left( y_jI  - A_j\right ) \ \ \
 \text{ where } \ \ 
 A_j = (-1)^{n-j-1} \cdot  \sum_{i=1}^d y_i \cdot \varphi \wedge w_j \wedge v_i .
\]
The matrices $A_{d+1}, \hdots, A_n$ commute and represent the action of $y_{d+1}, \hdots, y_n$ on $M$. 
Finishing the construction above requires taking a basis of $\R[\L^{-1}]$ as $R$-module.  For generic linear spaces 
$\L\in \Gr(d,n)$, we can using the following. 

\begin{prop}\label{prop:coordBasis}
Suppose the Pl\"ucker coordinates of $\L$ are all non-zero. The monomials of degree at most $d-1$ in variables $y_{d+1}, \hdots, y_n$ 
form an $R$-basis of $\R[\L^{-1}]$.
\end{prop}
\begin{proof}
Since the matroid associated to $\L$ is the uniform matroid of rank $d$, every circuit has size $d+1$.  It follows that every polynomial in the ideal $\mathcal{I}(\L^{-1})\subset \R[y_1, \hdots, y_n]$
of polynomials vanishing on $\L^{-1}$ has degree $\geq d$. Then the ideal $\mathcal{I}(\L^{-1})+(y_1,\ldots,y_d)$ does not contain any polynomials in $\R[y_{d+1},\hdots, y_n]$
of degree $\leq d-1$.  From this, we see that the monomials of degree at most $d-1$ in $y_{d+1},\hdots, y_n$ are linearly independent in the $\R$-vector space 
$\R[\L^{-1}]/(y_1, \hdots, y_d)$. The dimension of this $\R$-vector space is the same as the rank of the free $R$-module $\R[\L^{-1}]$ which is $\binom{n-1}{d-1}$. This is also the number of monomials of degree at most $d-1$ in $n-d$ variables. Thus, the monomials of degree at most $d-1$ in $y_{d+1},\hdots, y_n$ are in fact a basis of the $\R$-vector space $\R[\L^{-1}]/(y_1, \hdots, y_d)$. By Nakayama's Lemma, this implies that these monomials generate $\R[\L^{-1}]$ as a free $R$-module. Thus, again by comparing rank and number of monomials the claim follows. \end{proof}

Now we assume that all Pl\"ucker coordinates of $\L$ are non-zero, let $k = \binom{n-1}{d-1} = \deg(\L^{-1})$, and take the 
monomials $\{(y_{d+1})^{\alpha_1} \cdots (y_n)^{\alpha_{n-d}}\}_{\alpha\in \Delta}$ as the basis of 
$\R[\L^{-1}]$ as an $R$-module, where $\Delta$ denotes the set of points in $(\Z_{\geq 0})^{n-d}$ whose coordinates sum to at most $d-1$. 
This writes the trace bilinear form of \eqref{eq:traceForm} as
\[
H  \ \ = \ \ 
\left( \tr\left((A_{d+1})^{\alpha_1 + \beta_1} \cdots (A_{n})^{\alpha_{n-d} + \beta_{n-d}} \right) \right)_{\alpha, \beta \in \Delta}
\ \ = \ \
\left( \tr(B_{\alpha} \cdot B_{\beta}) \right)_{\alpha,  \beta \in \Delta}
\]
where for any $\alpha\in \Delta$, $B_{\alpha}$ equals the matrix product $(A_{d+1})^{\alpha_1} \cdots (A_{n})^{\alpha_{n-d}} $.
Then $H$ equals the product $UU^T$ where $U$ is the $k\times k^2$ matrix
whose rows consist of the vectors $\vecz(B_{\alpha})$ for $\alpha \in \Delta$. 
The entropic discriminant is the determinant of $H$. By the Cauchy-Binet theorem, 
this is a sum of the squares of the $\binom{k^2}{k}$ maximal minors of $U$.

\begin{example}\label{exp:entrsos}
 We illustrate this construction on a small example.
 Let $\L\in \Gr(2,4)$ be the rowspan of $\begin{pmatrix} 1 & 1 & 1 & 1 \\ 0 & 1 & 2 & 3 \end{pmatrix}$.
The entropic discriminant is the  the discriminant of the projection $\mathcal{L}^{-1} \to \mathbb{P}^1, 
(x_1:x_2:x_3:x_4) \to (x_1+x_2 + x_3+x_4:x_2+2 x_3+3x_4)$
 from $\mathcal{L}^{\perp}$. As in Example~\ref{ex:Gr24}, the determinantal representation
 of $\L^{-1}$ that we get from our construction is 
 \[\varphi =          \begin{pmatrix}
 e_{12}+\frac{1}{2} e_{13}+\frac{1}{3} e_{14} & -e_{12} & -\frac{1}{2}  e_{13} \\
 -e_{12} & e_{12}+e_{23}+\frac{1}{2} e_{24} & -e_{23} \\
 -\frac{1}{2} e_{13} & -e_{23} & \frac{1}{2} e_{13}+e_{23}+e_{34} \\
\end{pmatrix}  \ \ \text{ in }  \ \ \bigwedge\nolimits^2 \R^4 \otimes \Sym_3(\R).
            \]
Here we take $v_1=e_1+e_2+e_3+e_4$, $v_2=e_2+2 e_3+3e_4$, $v_3=-e_1+2e_2-e_3$ and $v_4=-e_2+2e_3-e_4$, 
and $y_i = \sum_j v_{ij} x_j$. 
Since $\varphi$ is definite at $\L^{\perp}$, $\varphi\wedge v_3 \wedge v_4$ is positive definite. Indeed, 
\[ \varphi \wedge v_3 \wedge v_4 
\ \ = \ \ \begin{pmatrix} 3 & -1 & -1 \\ -1 & 3 & -1 \\ -1 & -1 & 3\end{pmatrix}
\ \ =  \ \ QQ^T 
\ \ \text{ where }  \ \ 
Q =  \begin{pmatrix} 1 & 1 & -1 \\ 1 & -1 & 1 \\ -1 & 1 & 1\end{pmatrix}.
\]            
Then for $w_3 = v_4$ and $w_4 = v_3$, the matrix 
$Q^{-1}(\varphi\wedge w_j)Q^{-T}$ can be written as $\pm \left( y_jI  - A_j\right )$ where for $j=3,4$, 
$A_j $ equals $  (-1)^{j-1} \cdot Q^{-1} (y_1\cdot  \varphi\wedge w_j \wedge v_1  + y_2 \cdot \varphi\wedge w_j \wedge v_2 )Q^{-T}$, giving
\begin{align*}
&\hspace{.2in}A_3 =   &&\hspace{.2in}A_4 = \\ 
 &{\footnotesize \frac{1}{12}
\begin{pmatrix}
 3 y_1-y_2 & -6 y_1-4 y_2 & 3 y_1-3 y_2 \\
 -6 y_1-4 y_2 & 12 y_1+20 y_2 & -6 y_1-24 y_2 \\
 3 y_1-3 y_2 & -6 y_1-24 y_2 & 15 y_1+27 y_2 \\    \end{pmatrix} 
}, &&{\footnotesize
\frac{1}{12}
\begin{pmatrix}
 3 y_1+10 y_2 & -6 y_1-14 y_2 & -3 y_1-12 y_2 \\
 -6 y_1-14 y_2 & 12 y_1+16 y_2 & 6 y_1-6 y_2 \\
 -3 y_1-12 y_2 & 6 y_1-6 y_2 & 15 y_1+18 y_2 \\
\end{pmatrix} }. \smallskip \\
\end{align*}
By Proposition~\ref{prop:coordBasis}, $\{1,y_3, y_4\}$ forms a basis for $\R[\L^{-1}]$ over $R = \R[y_1, y_2]$. 
Therefore the trace bilinear form of \eqref{eq:traceForm} is represented by the matrix $H =$
\[ \begin{pmatrix}
     \tr( \textrm{I}_3 ) & \tr(A_3) & \tr(A_4) \\
     \tr(A_3) & \tr(A_3^2) & \tr(A_3 A_4) \\
     \tr(A_4) & \tr(A_3 A_4) & \tr(A_4^2)
    \end{pmatrix} =  
\begin{pmatrix}3 & \frac{5 y_1}{2}+\frac{23 y_2}{6} & \frac{5 y_1}{2}+\frac{11
   y_2}{3} \\
 \frac{5 y_1}{2}+\frac{23 y_2}{6} & \frac{15 y_1^2}{4}+\frac{40
   y_1 y_2}{3}+\frac{583 y_2^2}{36} & \frac{5 y_1^2}{2}+\frac{15
   y_1 y_2}{2}+\frac{317 y_2^2}{36} \\
 \frac{5 y_1}{2}+\frac{11 y_2}{3} & \frac{5 y_1^2}{2}+\frac{15
   y_1 y_2}{2}+\frac{317 y_2^2}{36} & \frac{15 y_1^2}{4}+\frac{55
   y_1 y_2}{6}+\frac{179 y_2^2}{18} \\
   \end{pmatrix}.
\]
The determinant of $H$ is the entropic discriminant: 
\[\det(H)  \ \ = \ \ \frac{25}{144} \left(45 y_1^4+270 y_1^3 y_2+763 y_1^2
   y_2^2+1074 y_1 y_2^3+773 y_2^4\right).\]
As the matrices $A_3, A_4$ are symmetric, $H$ can be written as the product $UU^T$, 
where $U$ is the $3\times 9$ matrix with rows $\vecz(I)$, $\vecz(A_3)$ and $\vecz(A_4)$. 
Using the Cauchy-Binet Theorem, this writes the entropic discriminant is a sum of $\binom{9}{3}$ squares. 
This is far from the shortest sum of squares representation since any nonnegative binary form is a sum of two squares. 
\end{example}


\begin{thebibliography}{10}

\bibitem{AK70}
A.~Altman and S.~Kleiman.
\newblock {\em Introduction to {G}rothendieck duality theory}.
\newblock Lecture Notes in Mathematics, Vol. 146. Springer-Verlag, Berlin-New
  York, 1970.

\bibitem{scattering}
N.~Arkani-Hamed, J.~Bourjaily, F.~Cachazo, A.~Goncharov, A.~Postnikov, and
  J.~Trnka.
\newblock {\em Grassmannian geometry of scattering amplitudes}.
\newblock Cambridge University Press, Cambridge, 2016.

\bibitem{BK}
O.~Bernardi and C.~Klivans.
\newblock Directed rooted forests in higher dimension, Preprint, available at
  \url{http://arxiv.org/abs/1512.07757}, 2015.

\bibitem{Hadamard}
C.~Bocci, E.~Carlini, and J.~Kileel.
\newblock Hadamard products of linear spaces.
\newblock {\em J. Algebra}, 448:595--617, 2016.

\bibitem{stabilityPreservers}
J.~Borcea and P.~Br{\"a}nd{\'e}n.
\newblock The {L}ee-{Y}ang and {P}\'olya-{S}chur programs. {I}. {L}inear
  operators preserving stability.
\newblock {\em Invent. Math.}, 177(3):541--569, 2009.

\bibitem{Bor1846}
C.~W. Borchardt.
\newblock Neue {E}igenschaft der {G}leichung, mit deren {H}\"ulfe man die
  secul\"aren {S}t\"orungen der {P}laneten bestimmt.
\newblock {\em J. Reine Angew. Math.}, 30:38--45, 1846.

\bibitem{BrandenHPP}
P.~Br{\"a}nd{\'e}n.
\newblock Polynomials with the half-plane property and matroid theory.
\newblock {\em Adv. Math.}, 216(1):302--320, 2007.

\bibitem{Bra11}
P.~Br{\"a}nd{\'e}n.
\newblock Obstructions to determinantal representability.
\newblock {\em Adv. Math.}, 226(2):1202--1212, 2011.

\bibitem{Ulrich1}
J.~P. Brennan, J.~Herzog, and B.~Ulrich.
\newblock Maximally generated {C}ohen-{M}acaulay modules.
\newblock {\em Math. Scand.}, 61(2):181--203, 1987.

\bibitem{COSW04}
Y.-B. Choe, J.~G. Oxley, A.~D. Sokal, and D.~G. Wagner.
\newblock Homogeneous multivariate polynomials with the half-plane property.
\newblock {\em Adv. in Appl. Math.}, 32(1-2):88--187, 2004.
\newblock Special issue on the Tutte polynomial.

\bibitem{Chow}
W.-L. Chow and B.~L. van~der Waerden.
\newblock Zur algebraischen {G}eometrie. {IX}.
\newblock {\em Math. Ann.}, 113(1):692--704, 1937.

\bibitem{CL87}
R.~Cordovil and B.~Lindstr{\"o}m.
\newblock Simplicial matroids.
\newblock In {\em Combinatorial geometries}, volume~29 of {\em Encyclopedia
  Math. Appl.}, pages 98--113. Cambridge Univ. Press, Cambridge, 1987.

\bibitem{ChowIntro}
J.~Dalbec and B.~Sturmfels.
\newblock Introduction to {C}how forms.
\newblock In {\em Invariant methods in discrete and computational geometry
  ({C}ura\c cao, 1994)}, pages 37--58. Kluwer Acad. Publ., Dordrecht, 1995.

\bibitem{DLSV12}
J.~A. De~Loera, B.~Sturmfels, and C.~Vinzant.
\newblock The central curve in linear programming.
\newblock {\em Found. Comput. Math.}, 12(4):509--540, 2012.

\bibitem{ESW03}
D.~Eisenbud, F.-O. Schreyer, and J.~Weyman.
\newblock Resultants and {C}how forms via exterior syzygies.
\newblock {\em J. Amer. Math. Soc.}, 16(3):537--579, 2003.

\bibitem{Gar51}
L.~G{\aa}rding.
\newblock Linear hyperbolic partial differential equations with constant
  coefficients.
\newblock {\em Acta Math.}, 85:1--62, 1951.

\bibitem{Guel97}
O.~G{\"u}ler.
\newblock Hyperbolic polynomials and interior point methods for convex
  programming.
\newblock {\em Math. Oper. Res.}, 22(2):350--377, 1997.

\bibitem{Gurvits}
L.~Gurvits.
\newblock Combinatorial and algorithmic aspects of hyperbolic polynomials.
\newblock {\em Electronic Colloquium on Computational Complexity {(ECCC)}},
  (070), 2004.

\bibitem{hanselka2014definite}
C.~Hanselka.
\newblock Definite determinantal representations of ternary hyperbolic forms,
  Preprint, available at \url{http://arxiv.org/abs/1411.1661}, 2014.

\bibitem{Hatcher}
A.~Hatcher.
\newblock {\em Algebraic topology}.
\newblock Cambridge University Press, Cambridge, 2002.

\bibitem{HV07}
J.~W. Helton and V.~Vinnikov.
\newblock Linear matrix inequality representation of sets.
\newblock {\em Comm. Pure Appl. Math.}, 60(5):654--674, 2007.

\bibitem{Ulrich2}
J.~Herzog, B.~Ulrich, and J.~Backelin.
\newblock Linear maximal {C}ohen-{M}acaulay modules over strict complete
  intersections.
\newblock {\em J. Pure Appl. Algebra}, 71(2-3):187--202, 1991.

\bibitem{Ily92}
N.~V. Ilyushechkin.
\newblock The discriminant of the characteristic polynomial of a normal matrix.
\newblock {\em Mat. Zametki}, 51(3):16--23, 143, 1992.

\bibitem{kummer2013determinantal}
M.~Kummer.
\newblock Determinantal representations and b\'ezoutians.
\newblock {\em Math. Z.}, 2016.

\bibitem{us}
M.~Kummer, D.~Plaumann, and C.~Vinzant.
\newblock Hyperbolic polynomials, interlacers, and sums of squares.
\newblock {\em Math. Program.}, 153(1, Ser. B):223--245, 2015.

\bibitem{KuSh}
M.~Kummer and E.~Shamovich.
\newblock Real fibered morphisms and ulrich sheaves, Preprint, available at
  \url{http://arxiv.org/abs/1507.06760}, 2015.

\bibitem{Lax98}
P.~D. Lax.
\newblock On the discriminant of real symmetric matrices.
\newblock {\em Comm. Pure Appl. Math.}, 51(11-12):1387--1396, 1998.

\bibitem{lusztig}
G.~Lusztig.
\newblock Total positivity in reductive groups.
\newblock In {\em Lie theory and geometry}, volume 123 of {\em Progr. Math.},
  pages 531--568. Birkh\"auser Boston, Boston, MA, 1994.

\bibitem{MSS14}
A.~W. Marcus, D.~A. Spielman, and N.~Srivastava.
\newblock Interlacing families {II}: {M}ixed characteristic polynomials and the
  {K}adison-{S}inger problem.
\newblock {\em Ann. of Math. (2)}, 182(1):327--350, 2015.

\bibitem{MSUZ14}
M.~Micha{\l}ek, B.~Sturmfels, C.~Uhler, and P.~Zwiernik.
\newblock Exponential varieties.
\newblock {\em Proc. Lond. Math. Soc. (3)}, 112(1):27--56, 2016.

\bibitem{NPT13}
T.~Netzer, D.~Plaumann, and A.~Thom.
\newblock Determinantal representations and the {H}ermite matrix.
\newblock {\em Michigan Math. J.}, 62(2):407--420, 2013.

\bibitem{TimundThom12}
T.~Netzer and A.~Thom.
\newblock Polynomials with and without determinantal representations.
\newblock {\em Linear Algebra Appl.}, 437(7):1579--1595, 2012.

\bibitem{New72}
M.~J. Newell.
\newblock On identities associated with a discriminant.
\newblock {\em Proc. Edinburgh Math. Soc. (2)}, 18:287--291, 1972/73.

\bibitem{Oxley}
J.~Oxley.
\newblock {\em Matroid theory}, volume~21 of {\em Oxford Graduate Texts in
  Mathematics}.
\newblock Oxford University Press, Oxford, second edition, 2011.

\bibitem{PV13}
D.~Plaumann and C.~Vinzant.
\newblock Determinantal representations of hyperbolic plane curves: an
  elementary approach.
\newblock {\em J. Symbolic Comput.}, 57:48--60, 2013.

\bibitem{PS}
N.~Proudfoot and D.~Speyer.
\newblock A broken circuit ring.
\newblock {\em Beitr\"age Algebra Geom.}, 47(1):161--166, 2006.

\bibitem{Ren06}
J.~Renegar.
\newblock Hyperbolic programs, and their derivative relaxations.
\newblock {\em Found. Comput. Math.}, 6(1):59--79, 2006.

\bibitem{SSV13}
R.~Sanyal, B.~Sturmfels, and C.~Vinzant.
\newblock The entropic discriminant.
\newblock {\em Adv. Math.}, 244:678--707, 2013.

\bibitem{SV}
E.~Shamovich and V.~Vinnikov.
\newblock Liv{\v{s}}ic-type determinantal representations and hyperbolicity,
  Preprint, available at \url{http://arxiv.org/abs/1410.2826}, 2014.

\bibitem{Var95}
A.~Varchenko.
\newblock Critical points of the product of powers of linear functions and
  families of bases of singular vectors.
\newblock {\em Compositio Math.}, 97(3):385--401, 1995.

\bibitem{vppf}
V.~Vinnikov.
\newblock L{MI} representations of convex semialgebraic sets and determinantal
  representations of algebraic hypersurfaces: past, present, and future.
\newblock In {\em Mathematical methods in systems, optimization, and control},
  volume 222 of {\em Oper. Theory Adv. Appl.}, pages 325--349.
  Birkh\"auser/Springer Basel AG, Basel, 2012.

\bibitem{wagnerSurvey}
D.~G. Wagner.
\newblock Multivariate stable polynomials: theory and applications.
\newblock {\em Bull. Amer. Math. Soc. (N.S.)}, 48(1):53--84, 2011.

\bibitem{WhiteMatroids}
N.~White, editor.
\newblock {\em Matroid applications}, volume~40 of {\em Encyclopedia of
  Mathematics and its Applications}.
\newblock Cambridge University Press, Cambridge, 1992.

\end{thebibliography}

\end{document}